\documentclass[12pt]{amsart}
\usepackage{fullpage, amssymb, xy, hyperref}
\xyoption{all}

\title[Cuspidal Representations of General Linear Groups]{On Cuspidal Representations of General Linear Groups over Discrete Valuation Rings}

\author[A.-M.~Aubert]{Anne-Marie Aubert}
\address{C.N.R.S., Institut de Math\'ematiques de Jussieu, 175 rue du Chevaleret, 75013 Paris, France.}

\author[U.~Onn]{Uri Onn$^{~\star}$}
\thanks{$\star$ Supported by the Israel Science Foundation, ISF
grant no. 555104, by the Edmund Landau Minerva Center for Research
in Mathematical Analysis and Related Areas, sponsored by the Minerva
Foundation (Germany).}
\address{Department of Mathematics, Ben Gurion University of the Negev, Be'er Sheva 84105, Israel.}

\author[A.~Prasad]{Amritanshu Prasad}
\address{The Institute of Mathematical Sciences, CIT campus, Taramani, Chennai 600 113, India.}

\author[A.~Stasinski]{Alexander Stasinski${}^\dagger$}
\address{DPMMS, University of Cambridge, Wilberforce Road, Cambridge, CB3
0WB, U.~K.}\thanks{${}^\dagger$ Supported at various times by EPSRC Grants GR/T21714/01 and EP/C527402.}

\subjclass[2000]{22E50, 11S37}
\keywords{Cuspidal representations; general linear groups; local rings}

\numberwithin{equation}{section}

\theoremstyle{remark}

\newtheorem*{remark*}{Remark}
\theoremstyle{plain}
\newtheorem{theorem}[equation]{Theorem}

\newtheorem{prop}[equation]{Proposition}
\newtheorem{lemma}[equation]{Lemma}
\newtheorem*{theorema}{Theorem~A}
\newtheorem*{theoremb}{Theorem~B}
\newtheorem*{theoremc}{Theorem~C}
\newtheorem*{theoremd}{Theorem~D}
\newtheorem*{theoreme}{Theorem~E}

\newtheorem*{question*}{Question}
\theoremstyle{definition}
\newtheorem{defn}[equation]{Definition}

 \theoremstyle{plain}
\newtheorem{thm}{Theorem}[section]
  \theoremstyle{remark}
  \newtheorem*{acknowledgement*}{Acknowledgement}
  \theoremstyle{plain}
  \newtheorem{lem}[thm]{Lemma}

\DeclareMathOperator{\Stab}{Stab}
\newcommand\leftexp[2]{\hspace{1pt}{\vphantom{#2}}^#1\hspace{-1pt}#2}
\newcommand{\scr}{\mathcal}
\newcommand{\germ}{\mathfrak}

\newcommand{\Aut}{\mathrm{Aut}}
\newcommand{\C}{\mathbb{C}}
\newcommand{\embed}[2]{{#1\hookrightarrow #2}}
\newcommand{\fA}{\mathfrak{A}}

\newcommand{\F}{\mathbb{F}}
\newcommand{\Fq}{\F_q}
\newcommand{\Gal}{\mathrm{Gal}}
\newcommand{\GL}{\mathrm{GL}}
\newcommand{\Hom}{\mathrm{Hom}}
\newcommand{\I}{\mathrm{I}}
\newcommand{\Ker}{\mathrm{Ker}}
\newcommand{\Ind}[2]{\mathrm{Ind}_{#1}^{#2}}
\newcommand{\cInd}[2]{\mathrm{c}\text{-}\mathrm{Ind}_{#1}^{#2}}
\newcommand{\inv}{^{-1}}
\newcommand{\linkn}{\embed{\lambda}{k^n}}
\newcommand{\lonkn}{\surj{k^n}{\lambda}}
\newcommand{\M}{M}
\newcommand{\N}{\mathbb{N}}
\newcommand{\Oh}{\mathfrak{o}}
\newcommand{\OO}{\mathfrak{O}}
\newcommand{\pp}{\mathfrak{p}}
\newcommand{\PP}{\mathfrak{P}}
\newcommand{\Rep}[1]{\mathrm{Rep}(#1)}
\newcommand{\re}{regular elliptic }
\newcommand{\surj}[2]{{#1\twoheadrightarrow #2}}
\newcommand{\tr}{\mathrm{tr}}
\newcommand{\tOh}{\OO}
\newcommand{\tok}{\tOh_k^\times}

\newcommand{\p}{\varpi}
\renewcommand{\vec}{\mathbf}

\newcommand{\T}{\p}
\renewcommand{\mathbb}{\mathbf}
\begin{document}

\begin{abstract}
  We define a new notion of cuspidality for representations of $\GL_n$ over a finite quotient $\Oh_k$ of the ring of integers $\Oh$ of a non-Archimedean local field $F$ using geometric and infinitesimal induction functors, which involve automorphism groups $G_\lambda$ of torsion $\Oh$\nobreakdash-modules.
  When $n$ is a prime, we show that this notion of cuspidality is equivalent to strong cuspidality, which arises in the construction of supercuspidal representations of $\GL_n(F)$.
  We show that strongly cuspidal representations share many features of cuspidal representations of finite general linear groups.
  In the function field case, we show that the construction of the representations of $\GL_n(\Oh_k)$ for $k\geq 2$ for all $n$ is equivalent to the construction of the representations of all the groups $G_\lambda$.
  A functional equation for zeta functions for representations of $\GL_n(\Oh_k)$ is established for representations which are not contained in an infinitesimally induced representation.
  All the cuspidal representations for $\GL_4(\Oh_2)$ are constructed.
  Not all these representations are strongly cuspidal.
\end{abstract}

\maketitle

%%%%%%%%%%%%%%%%%%%%%%%%%%%%%%%%%%%%%%%%%%%%%%%%%%%%%%%%%%%%%%%%%%%%%%%%%%%%%
%         Introduction                                                      %
%%%%%%%%%%%%%%%%%%%%%%%%%%%%%%%%%%%%%%%%%%%%%%%%%%%%%%%%%%%%%%%%%%%%%%%%%%%%%

\section{Introduction}
\label{sec:intro}

The irreducible characters of $\GL_n(\Fq)$ were computed by J.~A.~Green in 1955 \cite{MR0072878}.
In Green's work, parabolic induction was used to construct many irreducible characters of $\GL_n(\Fq)$ from irreducible characters of smaller general linear groups over $\Fq$.
The representations which could not be obtained in this way, known as \emph{cuspidal representations}, were shown to be in canonical bijective correspondence with Galois orbits of norm-primitive characters of $\F_{q^n}^\times$ (these are characters which do not factor through the norm map $\F_{q^n}^\times\to \F_{q^d}^\times$ for any proper factor $d$ of $n$).

Let $F$ be a non-Archimedean local field with ring of integers $\Oh$.
Let $\pp$ be the maximal ideal in $\Oh$, and $\Oh_k=\Oh/\pp^k$ for $k\geq 1$.
Thus $\Oh_1$ is a finite field, the residue field of $F$, which we take to be $\Fq$.
In contrast with $\GL_n(\Oh_1)$, not much is known in general about the representation theory of $\GL_n(\Oh_k)$.
Unlike general linear groups over fields, for which conjugacy classes are parameterized by Jordan canonical forms, the classification of conjugacy classes in $\GL_n(\Oh_k)$ for all $n$ and any $k\geq 2$ contains the matrix pair problem \cite[Section~4]{Nag78}, which is a wild classification problem.

The representations of $\GL_n(\Oh)$ received considerable attention after supercuspidal representations of $\GL_n(F)$ were constructed by induction from a compact-modulo-center subgroup \cite{MR0233931,MR0492087,MR0507253}.
A class of representations (\emph{repr\'esentations tr\`es cuspidales}) of the maximal compact-modulo-center subgroups
which give rise to irreducible supercuspidal representations of $\GL_n(F)$ were identified by Carayol \cite{MR760676}.
When the maximal compact subgroup modulo center in question is $F^\times \GL_n(\Oh)$, the restrictions of these representations to $\GL_n(\Oh)$ correspond to what we call \emph{strongly cuspidal representations} of $\GL_n(\Oh_k)$ for some $k$ (Definition~\ref{defn:Vcuspidality}).
Carayol used these representations to construct all the supercuspidal representations of $\GL_n(F)$ when $n$ is prime.
The classification of supercuspidal representations of $\GL_n(F)$ for $n$
arbitrary was completed by Bushnell and Kutzko \cite{MR1204652}.
Recently, Paskunas \cite{Paskunas} proved that given an irreducible
supercuspidal representation $\pi$ of $\GL_n(F)$, there exists a unique
(up to isomorphism) irreducible 
representation $\tau$ of $\GL_n(\Oh)$, such that $\tau$ is a type for the Bernstein component of $\pi$. 
Hence representations of $\GL_n(\Oh)$ occur
naturally in the representation theory of $\GL_n(F)$.
Nevertheless, with respect to $\GL_n(\Oh)$, since the general representation 
theory is unmanageably complicated, only those very
special representations that are needed to understand the representations
of the $p$\nobreakdash-adic group itself have been considered.

In this article, we take the point of view that the representation theory of $\GL_n(\Oh)$ is interesting in its own right, and while extremely complicated, does display a certain structure.
To this end, a new definition of cuspidality is introduced for representations of $\GL_n(\Oh_k)$.
This definition is closer in spirit to the characterization in \cite{MR0072878} of cuspidal representations as those which do not occur in representations obtained by parabolic induction.
More specifically, let $\Lambda$ denote the set of all partitions of all positive integers.
The isomorphism classes of finitely generated torsion $\Oh$-modules are parameterized by $\Lambda$.
For any $\Oh$-module $\Oh_\lambda=\oplus_{i=1}^{m}\Oh_{\lambda_i}$ of type $\lambda=(\lambda_1, \ldots, \lambda_m)\in \Lambda$, let $G_{\lambda}=G_{\lambda,F}$ stand for its automorphism group.
Thus, for example, $G_{k^n}=\GL_n(\Oh_k)$.
Say that $\lambda\leq \mu$ if $\Oh_\lambda$ can be embedded in $\Oh_\mu$.
We call an irreducible representation  of $G_{k^n}$ {\em cuspidal} (see Definition~\ref{defn:cuspidality}) if it cannot be constructed from lower building blocks.
By lower building blocks we mean the representations of $G_\lambda$, where $\lambda< k^n$.
These automorphism groups play the role of Levi components of proper parabolic subgroups of $\GL_n(\Oh_1)$.
Representations of $G_{k^n}$ are constructed from those of $G_\lambda$ using \emph{infinitesimal} and \emph{geometric} induction (Section~\ref{sec:functors}).
Our first result, which is proved in Section~\ref{sec:comparison}, compares cuspidality with strong cuspidality.
\begin{theorema}
  Every strongly cuspidal representation is cuspidal. When $n$ is prime every cuspidal representation is strongly cuspidal.
\end{theorema}
When $n$ is not prime, it is not true that every cuspidal representation is strongly cuspidal.
In Section~\ref{sec:alex}, all the cuspidal representations of $\GL_4(\Oh_2)$ are constructed. 
Among these are representations which are not strongly cuspidal.

The construction of strongly cuspidal representations is well-known \cite{MR0233931,MR0396859,MR0492087}.
When $n$ is prime, then by Theorem~A, all cuspidal representations are obtained in this manner.
Moreover, for all $n$, the strongly cuspidal representations have properties analogous to cuspidal representations of $\GL_n(\Oh_1)$.
Firstly, they can be parameterized in an analogous fashion.
Suppose that $E$ is an unramified extension of $F$ of degree $n$, and $\OO$ is the integral closure of $\Oh$ in $E$.
Let $\PP$ denote the maximal ideal in $\OO$ and $\OO_k$ denote the finite quotient ring $\OO/\PP^k$.
For $k>1$, a character $\tok\to \C^\times$ is said to be \emph{strongly primitive} if its restriction to $\ker(\tok\to \OO_{k-1}^\times)\cong \OO_1$ does not factor through any proper subfield via the trace map.
A character of $\OO_1^\times$ is said to be strongly primitive if it is norm-primitive.
In Section~\ref{sec:Green} we prove
\begin{theoremb}
  There is a canonical bijective correspondence between strongly cuspidal representations of $G_{k^n}$ and $\Gal(E/F)$-orbits of strongly primitive characters of $\tok$.
\end{theoremb}
The group of units $\tOh_k^{\times}$ is embedded in $G_{k^n}$ since $\tOh_k \simeq \Oh^n_k$ as $\Oh_k$-modules. 
An element of $G_{k^n}$ is said to be \re if it is conjugate to an element of $\tOh_k^{\times}$ whose image in $\tOh_1^{\times}$ lies in no proper subfield. 
In section~\ref{sec:construction}, we establish another property that strongly cuspidal representations share with cuspidal representations of $\GL_n(\Oh_1)$, which is that the correspondence of Theorem~B is well-behaved with respect to character values on \re elements.
\begin{theoremc}
  Let $\omega$ be a strongly primitive character of $\tOh_k^{\times}$ and let $\Theta_{\omega}$ be the corresponding strongly cuspidal character of $G_{k^n}$. Then for all \re elements $u \in \tOh_k^{\times} \subset G_{k^n}$
  \[
  \Theta_{\omega}(u)= (-1)^{(n-1)k} \sum_{\gamma \in \Gal(E/F)} \omega({^{\gamma}u}),
  \]
 Moreover, $\Theta_{\omega}$ vanishes on conjugacy classes which do not intersect $\tOh_k^{\times} \cdot  \Ker\{G_{k^n} \to G_{\lceil k/2 \rceil ^n} \}$.
\end{theoremc}
\begin{remark*}
  Theorems~B and~C are due to Green when $k=1$.
  For $k>1$, the ideas used in the proofs can be found in the existing literature on supercuspidal representations of $\GL_n(F)$, the detailed account in Section~\ref{sec:parameterization} gives the complete picture, working entirely inside $\GL_n(\Oh)$.
  In particular, Theorem~C is deduced from \cite[Theorem~1]{MR0396859}.
  It is closely related to the result obtained by Henniart in \cite[Section~3.7]{MR1263525}. 
We also observe that in \cite{MR2048585} Lusztig gave a geometric
construction of representations (in the function field case) which is likely to include the description of strongly cuspidal representations of $G_{k^n}$ in terms of strongly primitive characters.
\end{remark*}
There already is evidence that the representation theory of a group such as $G_\lambda$ can be studied by breaking up the problem into two parts.
The first is to correctly define and understand the cuspidal representations.
The second is to construct the remaining representations from cuspidal representations of $G_\mu$ with $\mu< \lambda$.
This approach has been implemented successfully in \cite{ranktwo} for automorphism groups of modules of rank two.
Theorems~A, B and C provide further evidence of the validity of this approach when $\lambda=k^n$ and $n$ is a prime.

The inevitability of the family of groups $G_\lambda$ in the representation theory of $G_{k^n}$ or even $G_{2^n}$ can be seen from another perspective. 
In Section~\ref{sec:complexity}, we prove
\begin{theoremd}
  Let $F$ be a local function field. Constructing the irreducible representations of the family of groups $\{ G_{2^n,F}=\GL_n(\Oh_2)~|~  n \in \N \}$ is equivalent to constructing the irreducible representations of the family $\{G_{\lambda,E}~|~  \lambda \in \Lambda,  ~E/F ~\text{unramified extension}\}$.
\end{theoremd}

Finally, we point out a suggestive connection to the Macdonald correspondence which might admit a higher level incarnation as well. 
Macdonald has established a correspondence between irreducible representations of $G_{1^n}$ and equivalence classes of $n$-dimensional tamely ramified representations of the Weil-Deligne group $W'_F$ \cite{MI4}.
One ingredient in this correspondence is a functional equation for the zeta function associated to $G_{1^n}$. 
It admits a straightforward generalization to $G_{k^n}$ for $k>1$. 
Let $\hat{f}$ denote a properly normalized additive Fourier transform of $f\in \C\left(\M_n(\Oh_k)\right)$ with respect to $\psi\left(\tr(\cdot)\right)$, where $\psi:\Oh_k \to \C$ is an additive character which does not factor through $\Oh_{k-1}$.
Let $\mathcal{Z}(f,\rho)=\sum_{g \in G_{k^n}} f(g)\rho(g) \in \text{End}_{\C}(V)$ where $f \in \C\left(\M_n(\Oh_k)\right)$ and $(\rho,V)$ is an irreducible representation of $G_{k^n}$. Denote by $\check{\rho}$ the contragredient representation of $\rho$. 
In Section~\ref{sec:zeta}, we prove

\begin{theoreme}
  If $\rho$ is not contained in an infinitesimally induced representation (in particular if $\rho$ is cuspidal), there exists a complex number $\varepsilon(\rho,\psi)$ and a such that
  \[
{^t\mathcal{Z}}(\hat{f},\check{\rho})=\varepsilon(\rho,\psi) \mathcal{Z}(f,\rho).
\]

\end{theoreme}

\subsection{Acknowledgments} The second author is grateful to Alex Lubotzky and Yakov Varshavsky for supporting this research. The third author acknowledges M.~K.~Vemuri for some very helpful discussions on Heisenberg groups. The second and third authors thank Shahar Mendelson, Amnon Neeman and the Australian National University in Canberra for giving them an opportunity to work together.
The fourth author thanks S.~Stevens for many helpful discussions which were instrumental for parts of the present work.
The authors thank Robert Kottwitz and Dipendra Prasad, who read a draft of this article and provided some very valuable feedback.
They are grateful to the referee for his careful reading of the article and valuable comments.

%%%%%%%%%%%%%%%%%%%%%%%%%%%%%%%%%%%%%%%%%%%%%%%%%%%%%%%%%%%%%%%%%%%%%%%%%%%%%
%         Generalities                                                      %
%%%%%%%%%%%%%%%%%%%%%%%%%%%%%%%%%%%%%%%%%%%%%%%%%%%%%%%%%%%%%%%%%%%%%%%%%%%%%

\section{Notations and preliminaries}
\label{sec:generalities}
\subsection{Automorphism groups}
\label{sec:some-groups}
Let $\Lambda$ denote the set of all partitions of all positive integers. Any $\lambda \in \Lambda$ can be written in the form  $(\lambda_1^{r_1},\ldots,\lambda_l^{r_l})$, where $\lambda_1>\cdots>\lambda_l$ and $r_1,\ldots,r_l$ are positive integers. The sum $r_1+\cdots +r_l$ will be called the \emph{length} of the partition, and $\lambda_1$ will be called the \emph{height} of the partition.

Every finitely generated torsion $\Oh$-module is of the form $\Oh_\lambda=\Oh_{\lambda_1}^{r_1}\oplus\cdots\oplus\Oh_{\lambda_l}^{r_l}$
for some $\lambda \in \Lambda$ of height no more than $k$. Consider the group $G_\lambda=\Aut_{\Oh}(\Oh_\lambda)$. In particular, taking $\lambda=(k^n)$, we have $G_{k^n}=\GL_n(\Oh_k)$.
When it is necessary to specify the underlying non-Archimedean local field $F$, the notation $G_{\lambda,F}$ will be used for $G_\lambda$. 

Let $N_r$ denote the kernel of the natural map $G_{k^n}\to G_{r^n}$.
Then, if $r\geq k/2$, the map $\M_n(\Oh_{k-r})\to N_r$ defined by
$A\mapsto \I+\p^r A$, is an isomorphism of groups (it is a bijection of sets for all $r<k$). This results in a short exact sequence
  \begin{equation}\label{eq:ses}
        0 \to \M_n(\Oh_{k-r}) \to G_{k^n} \to  G_{r^n} \to 1,
  \end{equation}
for every $r\geq k/2$.
In what follows, we identify $M_n(\Oh_{k-r})$ with its image in $G_{k^n}$ for $r\geq k/2$.

%Consider for a moment the special case where $r=1$.
%According to \cite[II.4, Proposition~8]{MR0354618}, the map $\Oh_k\to \Oh_1$ has
%a unique multiplicative section (which is also additive when $F$ is a local
%function field) $s\colon\Oh_1\to \Oh_k$.
%This can be used to define a section $s\colon G_{1^n}\to G_{k^n}$ by applying it to each entry of a matrix in $G_{1^n}$.
%The section $s$ is a homomorphism when $F$ is a function field.
\subsection{Similarity classes associated to representations}
\label{sec:simil-class-assoc}
Assume that $r\geq k/2$.
The action of $G_{k^n}$ on its normal subgroup $\M_n(\Oh_{k-r})$ factors through $G_{(k-r)^n}$.
In fact, this is just the usual action by similarity transformations
\begin{equation*}
  g\cdot A=gAg\inv, \quad g\in G_{(k-r)^n},\: A\in \M_n(\Oh_{k-r}).
\end{equation*}
It results in an action of $G_{(k-r)^n}$ on the set of all characters of $\M_n(\Oh_{k-r})$.

Now suppose that $\rho$ is an irreducible representation of $G_{k^n}$ on a vector space $V$.
The restriction of $\rho$ to $\M_n(\Oh_{k-r})$ gives rise to a decomposition $V=\oplus V_\chi$, where $\chi$ ranges over the set of characters of $\M_n(\Oh_{k-r})$.
Clifford theory then tells us that the set of characters $\chi$ for which $V_\chi$ is non-trivial consists of a single orbit for the action of $G_{(k-r)^n}$ on the characters of $\M_n(\Oh_{k-r})$.

The group $\M_n(\Oh_{k-r})$ can be identified with its Pontryagin dual (as a $G_{(k-r)^n}$-space).
For this, pick an additive character $\psi$ of $F\to \C^\times$ whose restriction to $\Oh$ is trivial, but whose restriction to $\pp^{-1}$ is non-trivial.
For each $A\in \M_n(\Oh_{k-r})$, define a character $\psi_A\colon\M_n(\Oh_{k-r})\to \C^\times$ by $\psi_A(B)=\psi(\p^{r-k}\tr(AB))$.
The map $A\mapsto \psi_A$ identifies $\M_n(\Oh_{k-r})$ with its Pontryagin dual, and preserves the action of $G_{(k-r)^n}$.

Thus we associate, for each $r\geq k/2$, to each irreducible representation $\rho$ of $G_{k^n}$, a similarity class $\Omega_{k-r}(\rho)\subset \M_n(\Oh_{k-r})$.

%%%%%%%%%%%%%%%%%%%%%%%%%%%%%%%%%%%%%%%%%%%%%%%%%%%%%%%%%%%%%%%%%%%%%%%%%%%%%
%         Induction and restriction functors                                %
%%%%%%%%%%%%%%%%%%%%%%%%%%%%%%%%%%%%%%%%%%%%%%%%%%%%%%%%%%%%%%%%%%%%%%%%%%%%%

\section{Induction and restriction functors}
\label{sec:functors}
This section introduces the functors that will play the role of parabolic induction and restriction in the context of $\GL_n(\Oh_k)$.
They were introduced in \cite[Section~2]{ranktwo}.
Geometric induction is an obvious analog of parabolic induction in the case of a field.
Infinitesimal induction has no analog in that setting.

\subsection{Geometric induction and restriction functors}
\label{subsec:geometric}
Given a direct sum decomposition
$\Oh_k^n=\Oh_k^{n_1}\oplus \Oh_k^{n_2}$, define $P_{n_1,n_2}$ to be the subgroup of $G_{k^n}$ which preserves $\Oh_k^{n_1}$.
There is a natural surjection $\varphi\colon P_{n_1,n_2}\to G_{k^{n_1}}\times G_{k^{n_2}}$.
Denote the kernel by $U_{n_1,n_2}$.
Define the functor $i_{n_1,n_2}\colon\Rep{G_{k^{n_1}}}\times \Rep{G_{k^{n_2}}}\to \Rep{G_{k^n}}$ taking representations $\sigma_1$ and $\sigma_2$ of $G_{k^{n_1}}$ and $G_{k^{n_2}}$ respectively to the induction to $G_{k^n}$ of the pull-back under $\varphi$ of $\sigma_1\otimes \sigma_2$.
The functor $r_{n_1,n_2}\colon\Rep{G_{k^n}}\to \Rep{G_{k^{n_1}}}\times \Rep{G_{k^{n_2}}}$ is defined by restricting a representation $\rho$ of $G_{k^n}$ to $P_{n_1,n_2}$ and then taking the invariants under $U_{n_1,n_2}$.
By Frobenius reciprocity, these functors form an adjoint pair:
\begin{equation*}
  \Hom_{G_{k^n}}(\rho,i_{n_1,n_2}(\sigma_1,\sigma_2))=\Hom_{G_{k^{n_1}}\times G_{k^{n_2}}}(r_{n_1,n_2}(\rho),\sigma_1\otimes\sigma_2).
\end{equation*}
Following \cite{ranktwo}, the functors $i_{n_1,n_2}$ and $r_{n_1,n_2}$ are called \emph{geometric induction} and \emph{geometric restriction} functors, respectively.
Furthermore
\begin{defn}
  An irreducible representation of $G_{k^n}$ will be said to lie in the geometrically induced series if it is isomorphic to a subrepresentation of $i_{n_1,n_n}(\sigma_1,\sigma_2)$ for some decomposition $n=n_1+n_2$ with $n_1$ and $n_2$ strictly positive, and some representations $\sigma_1$ and $\sigma_2$ of $G_{k^{n_1}}$ and $G_{k^{n_2}}$ respectively.
\end{defn}

\subsection{Infinitesimal induction and restriction functors}
\label{subsec:infinitesimal}
For two partitions $\lambda$ and $\mu$, say that $\lambda\leq \mu$ if there exists an embedding of $\Oh_\lambda$ in $\Oh_\mu$ as an $\Oh$-module.
This is equivalent to the existence of a surjective $\Oh$-module morphism $\Oh_\mu\to \Oh_\lambda$.
If $\lambda\leq k^n$, then the pair $(\lambda,k^n)$ has the \emph{unique embedding} and \emph{unique quotient} properties, i.e., all embeddings of $\Oh_\lambda$ in $\Oh_{k^n}$ and all surjections of $\Oh_{k^n}$ onto $\Oh_\lambda$ lie in the same $G_{k^n}$-orbit. As a consequence the functors that are defined below will, up to isomorphism, not depend on the choices of embeddings and surjections involved (in the language of \cite[Section 2]{MR2283434}, $k^n$ is a \emph{symmetric type}).

Given $\lambda \leq k^n$, take the obvious embedding of $\Oh_\lambda$ in $\Oh_k^n$ given on standard basis vectors by $\vec f_i\mapsto \pi^{k-\lambda_{h(i)}}\vec e_i$, where $h(i)$ is such that $r_1+\cdots +r_{h(i)-1}<i\leq r_1+\cdots + r_{h(i)}$.
Define
\begin{equation*}
  P_\linkn = \{ g\in G_{k^n}\;|\: g\cdot \Oh_\lambda = \Oh_\lambda\},
\end{equation*}
Restriction to $\Oh_\lambda$ gives rise to a homomorphism $P_\linkn\to G_\lambda$ which, due to the unique embedding property, is surjective.
Let $U_\linkn$ be the kernel.
One may now define an induction functor $i_\linkn\colon\Rep{G_\lambda}\to \Rep{G_{k^n}}$ as follows: given a representation of $G_\lambda$, pull it back to a representation of $P_\linkn$ via the homomorphism $P_\linkn\to G_\lambda$, and then induce to $G_{k^n}$.
Its adjoint functor $r_\linkn\colon\Rep{G_{k^n}}\to \Rep{G_\lambda}$ is obtained by taking a representation of $G_{k^n}$, restricting to $P_\linkn$, and taking the vectors invariant under $U_\linkn$.
The adjointness is a version of Frobenius reciprocity: there is a natural isomorphism
\begin{equation*}
  \Hom_{G_{k^n}}\left(\rho, i_\linkn (\sigma)\right)=\Hom_{G_\lambda}\left(r_\linkn (\rho), \sigma \right)
\end{equation*}
for representations $\rho$ and $\sigma$ of $G_{k^n}$ and $G_\lambda$ respectively.
In terms of matrices, the groups $P_\linkn$ and $U_\linkn$ are
\begin{gather*}
  P_\linkn = \{ (a_{ij})\in G_{k^n}\;|\:a_{ij}\in \pi^{\min\{0,\lambda_{h(j)}-\lambda_{h(i)}\}}\Oh_k\},\\
  U_\linkn = \{ (a_{ij})\in P_\linkn\;|\:a_{ij}\in \delta_{ij}+\pi^{\lambda_{h(j)}}\Oh_k\}.
\end{gather*}

Dually, fix the surjection of $\Oh_k^n$ onto $\Oh_\lambda$ given by $\vec e_i\mapsto \vec f_i$ and define
\begin{equation*}
  P_\lonkn = \{ g\in G_{k^n}\;|\: g \cdot \ker(\Oh_k^n\to \Oh_\lambda)=\ker(\Oh_k^n\to \Oh_\lambda)\}.
\end{equation*}
Taking the induced map on the quotient gives rise to a homomorphism $P_\lonkn\to G_\lambda$ which, by the unique quotient property, is surjective.
Let $U_\lonkn$ denote the kernel.
An adjoint pair of functors $i_\lonkn\colon\Rep{G_\lambda}\to \Rep{G_{k^n}}$ and $r_\lonkn\colon \Rep{G_{k^n}}\to \Rep{G_\lambda}$ are defined exactly as before.
$P_\lonkn$ is conjugate to $P_{\embed{\lambda'}{k^n}}$ and $U_\lonkn$ is conjugate to $U_{\embed{\lambda'}{k^n}}$, where $\lambda'$ is the partition that is complementary to $\lambda$ in $k^n$, i.e., the partition for which $\ker(\Oh_k^n\to \Oh_\lambda)\cong \Oh_{\lambda'}$. 
Therefore, the collection of irreducible representations obtained as summands after applying either of the functors $i_{\linkn}$ or $i_{\lonkn}$ is the same. 
Following \cite{ranktwo}, the functors $i_{\linkn}$ and $i_{\lonkn}$ are called \emph{infinitesimal induction functors}.
The functors $r_\linkn$ and $r_\lonkn$ are called \emph{infinitesimal restriction functors}. 
\begin{defn}
  An irreducible representation of $G_{k^n}$ will be said to lie in the infinitesimally induced series if it is isomorphic to a subrepresentation of $i_\linkn \sigma$ for some partition $\lambda\leq k^n$ and some representation $\sigma$ of $G_\lambda$.
\end{defn}

%%%%%%%%%%%%%%%%%%%%%%%%%%%%%%%%%%%%%%%%%%%%%%%%%%%%%%%%%%%%%%%%%%%%%%%%%%%%%
%         Cuspidal representations                                          %
%%%%%%%%%%%%%%%%%%%%%%%%%%%%%%%%%%%%%%%%%%%%%%%%%%%%%%%%%%%%%%%%%%%%%%%%%%%%%

\section{Cuspidality and strong cuspidality}
\label{sec:cuspidal}
\subsection{The definitions of cuspidality}
\label{sec:defns}
Recall from Section~\ref{sec:simil-class-assoc} that to every irreducible representation $\rho$ of $G_{k^n}$ is associated a similarity class $\Omega_1(\rho)\subset \M_n(\Oh_1)$. The following definition was introduced in \cite{MR592296} for $n=2$ and in \cite{MR760676} for general $n$.
\begin{defn}
  [Strong cuspidality]
  \label{defn:Vcuspidality}
  An irreducible representation $\rho$ of $G_{k^n}$ is said to be \emph{strongly cuspidal} if either $k=1$ and $\rho$ is cuspidal, or $k>1$ and $\Omega_1(\rho)$ is an irreducible orbit in $\M_n(\Oh_1)$.
\end{defn}
In the above definition, one says that an orbit is irreducible if the matrices in it are irreducible, i.e., they do not leave any non-trivial proper subspaces of $\Oh_1^n$ invariant.
This is equivalent to saying that the characteristic polynomial of any matrix in the orbit is irreducible.

Another notion of cuspidality (which applies for any $G_{\lambda}$, however, we shall focus on $\lambda=k^n$) picks out those irreducible representations which can not be constructed from the representations of $G_\lambda$, $\lambda \leq k^n$ by using the functors defined in Section \ref{sec:functors}. 
\begin{defn}
  [Cuspidality]
  \label{defn:cuspidality}
  An irreducible representation $\rho$ of $G_{k^n}$ is said to be \emph{cuspidal} if no twist of it by a linear character lies in the geometrically or infinitesimally induced series.
\end{defn}

\subsection{Comparison between the definitions}
\label{sec:comparison}
\begin{theorem}
  \label{theorem:equivalence}
  Every strongly cuspidal representation is cuspidal.
  When $n$ is a prime, every cuspidal representation is strongly cuspidal.
\end{theorem}

\begin{proof} Let $\rho$ be an irreducible non-cuspidal representation of $G_{k^n}$. The linear characters of $G_{k^n}$ are of the form $\det \!\circ \chi$ for some character $\chi\colon \Oh_k^{\times} \to \C^{\times}$. Using the identification of $N_{k-1} \simeq \M_n(\Oh_1)$ with its dual from Section \ref{sec:simil-class-assoc}, the restriction of $\det \!\circ \chi$ to $N_{k-1}$ is easily seen to be a scalar matrix.
Thus $\rho$ is strongly cuspidal if and only if $\rho(\chi)=\rho \otimes \det \!\circ \chi$ is, since adding a scalar matrix does not effect the irreducibility of the orbit $\Omega_1(\rho)$. Since $\rho$ is non-cuspidal, there exists a character $\chi$ such that $\rho(\chi)^U$ is nonzero for some $U=U_{n_1,n_2}$ or $U=U_{\lambda \hookrightarrow k^n}$. In either case this implies that the orbit $\Omega_1\left(\rho(\chi)\right)$ is reducible which in turn implies that $\rho(\chi)$ and hence $\rho$ are not strongly cuspidal.

For the converse the following interesting result (for which the hypothesis that $n$ is prime is not necessary) plays an important role.
A similar result was obtained by Kutzko in the context of supercuspidal representations of $\GL_n$ over a $p$-adic field \cite[Prop.~4.6]{MR0808102}.
Call a similarity class in $\M_n(\Oh_1)$ \emph{primary} if its characteristic polynomial has a unique irreducible factor.
\begin{prop}
  \label{prop:non-monatomic}
  Let $\rho$ be an irreducible representation of $G_{k^n}$.
  If $\Omega_1(\rho)$ is not primary then $\rho$ lies in the geometrically induced series.
\end{prop}
\begin{proof}
  If $\Omega_1(\rho)$ is not primary then it contains an element $\varphi=\left(\begin{smallmatrix} \hat{w}_1 & 0 \\0  & \hat{w}_2 \end{smallmatrix}\right)$ with $\hat{w}_i \in \M_{n_i}(\Oh_1)$ and $n=n_1+n_2$, such that the characteristic polynomials of $\hat{w}_1$ and $\hat{w}_2$ have no common factor.
  It will be shown that $r_{n_1,n_2}(\rho)\neq 0$.

  In what follows, matrices will be partitioned into blocks according to $n=n_1+n_2$.
  Let $P_i=P_{(k^{n_1},(k-i)^{n_2}) \hookrightarrow k^n}$ for $i=0,\ldots,k$.
  Then $P_i$ consists of matrices in $G_{k^n}$ with blocks of the form $\left(\begin{smallmatrix} a & b \\ \p^i c & d\end{smallmatrix}\right)$.
  Let $U_i$ be the normal subgroup of $P_i$ consisting of block matrices of the form $\left(\begin{smallmatrix} \I & \p^{k-i}u\\ 0 & \I\end{smallmatrix}\right)$. The $P_i$'s form a decreasing sequence of subgroups, while the $U_i$'s form increasing sequences. Given a representation $\rho_i$ of $P_i/U_i$ define $r_i(\rho_i)$ to be the representation of $P_{i+1}/U_{i+1}$ obtained by taking the vectors in the restriction of $\rho_i$ to $P_{i+1}$ that are invariant under $U_{i+1}$. That is,
  \[
r_i\colon\text{Rep}(P_i/U_i) \to \text{Rep}(P_{i+1}/U_{i+1}), \qquad  {r}_i(\rho_i) = \text{Inv}_{U_{i+1}/U_{i}} \circ \text{Res}^{P_i/U_i}_{P_{i+1}/U_{i}}(\rho_i).
\]
In particular, $P_k=P_{n_1,n_2}$ and $U_k=U_{n_1,n_2}$.
Therefore, (see \cite[Lemma 7.1]{ranktwo}) we have that $r_{n_1,n_2}=r_{k-1}\circ \cdots \circ r_0$. We argue by induction that $r_{i}\circ \cdots \circ r_0 (\rho) \ne 0$ for all $i=0,\ldots,k$. If $i=0$, then since $\varphi \in \Omega_1(\rho)$, we get that $\rho_{|U_1}$ contains the trivial character of $U_1$, hence, $r_0(\rho) \ne 0$. Denote $\rho_i=r_{i-1}\circ \cdots \circ r_0 (\rho)$ and assume that $\rho_i \ne 0$.  In order to show that $r_i(\rho_i) \ne 0$, consider the normal subgroup $L_i$ of $P_i$ which consists of block matrices of the form $\I+\big(\begin{smallmatrix} \p^{k-1}w_1& \p^{k-i-1}u \\ \p^{k-1}v & \p^{k-1}w_2 \end{smallmatrix}\big)$. It is easily verified that $L_{i}/U_{i} \simeq \M_n(\Oh_1)$, the isomorphism given by
\[
  \eta\colon \I+\left( \begin{matrix} \p^{k-1}w_1 & \p^{k-i-1}u \\ \p^{k-1}v  & \p^{k-1}w_2  \end{matrix} \right) \mod {U_{i}} ~~ \mapsto ~~\left( \begin{matrix} w_1 & u \\ v  & w_2  \end{matrix} \right),
\]
where $w_1, w_2, u$ and $v$ are appropriate block matrices over $\Oh_1$. It follows that we can identify the dual of ${L_{i}/U_{i}}$ with $\M_n(\Oh_1)$: $\hat{x} \mapsto \psi_{\hat{x}} \circ \eta$, for $\hat{x} \in \M_n(\Oh_1)$.

The action of $P_i$ on the dual of $L_i/U_i$ is given by $\hat{x} \mapsto g\hat{x}$ where $\psi_{g\hat{x}}(\eta(l)) =\psi_{\hat{x}}(\eta(g^{-1}l g))$. We shall not need the general action of elements of $P_i$, but rather of a small subgroup which is much easier to handle. If
\[
g_c=\left( \begin{matrix} \I &  \\ \p^{i}c  & \I  \end{matrix} \right), \qquad \eta(l)= \left( \begin{matrix} w_1 & u \\ v& w_2  \end{matrix} \right), \qquad \hat{x}= \left( \begin{matrix} \hat{w}_1 & \hat{v} \\ \hat{u}& \hat{w}_2  \end{matrix} \right),
\]
then unraveling definitions gives
\begin{equation}\label{dualaction}
\hat{x} \mapsto g_c \hat{x} = \left( \begin{matrix} \hat{w}_1 &  \hat{v} \\ \hat{u}+c\hat w_1-\hat w_2 c & \hat{w}_2  \end{matrix} \right).
\end{equation}
As we have identifications $L_0/U_1=\cdots =L_i/U_{i+1}$ we infer that the restriction of $\rho_i$ to $L_i/U_{i+1}$ contains a character
\[
\psi_{\hat{x}}=(\varphi_{|L_0/U_1},\hat{u})\colon L_i/U_{i+1} \times U_{i+1}/U_i = L_i/U_i \to \C^{\times},
\]
that is, $\psi_{\hat{x}}$ corresponds to $\hat{x}=\left( \begin{smallmatrix} \hat{w}_1 & 0 \\ \hat{u}& \hat{w}_2  \end{smallmatrix} \right)$. We claim that there exist $g_c$ such that
\[
g_c\hat{x}= \left( \begin{matrix} \hat{w}_1 & 0 \\ 0 & \hat{w}_2  \end{matrix} \right),
\]
therefore $\rho_{i|U_{i+1}/U_{i}}$ contains the trivial character of $U_{i+1}/U_{i}$ and hence $r_i(\rho_i) \ne 0$.

Indeed, using  \eqref{dualaction} it is enough to show that the map $c \mapsto c\hat{w}_1-\hat{w}_2 c$ is surjective, hence $\hat{u}$ can be eliminated and the entry $(1,2)$ contains the trivial character. This map is surjective if and only if it is injective. So we show that its kernel is null. A matrix $c$ is in the kernel if and only if
\begin{equation}\label{comm}
c \hat{w}_1= \hat{w}_2 c.
\end{equation}
Let $p_i$ ($i=1,2$) be the characteristic polynomials of $\hat{w}_i$. Our assumption on the orbits is that $p_1$ and $p_2$ have disjoint set of roots. Using \eqref{comm} we deduce that
\[
c p_1(\hat{w}_1)=p_1(\hat{w}_2)c.
\]
By the Cayley-Hamilton theorem the left hand side of the above equation vanishes. Over an algebraic closure of $\Oh_1$, $p_1(t)=\prod (t-\alpha_j)$, where the $\alpha_j$ are the roots of $p_1$. The hypothesis on $\hat{w}_1$ and $\hat{w}_2$ implies that none of these is an eigenvalue of $\hat{w}_2$. Therefore, $\hat{w}_2-\alpha_j$ is invertible for each $j$. It follows that $p_1(\hat{w}_2)=\prod (\hat{w}_2-\alpha_j)$ is also invertible, hence $c=0$. This completes the proof of the proposition.

\end{proof}

Returning now to the proof of Theorem~\ref{theorem:equivalence}, assume that $\rho$ is not strongly cuspidal. There are two possibilities:

 \begin{enumerate}

 \item [(a)] Any element $\hat{\omega} \in \Omega_1(\rho)$ has eigenvalue in $\Oh_1$. In such case, by twisting with a one-dimensional character $\chi$, we get a row of zeros in the Jordan canonical form of $\hat{\omega}$. Therefore, $\rho(\chi)$ is contained in a representation infinitesimally induced from $G_{(k^{n-1},k-1)}$.

 \item [(b)] Elements in $\Omega_1(\rho)$ have no eigenvalue in $\Oh_1$. Since $n$ is prime and since $\Omega_1(\rho)$ is reducible, the latter cannot be primary, and Proposition \ref{prop:non-monatomic} implies that $\rho$ lies in the geometrically induced series.
\end{enumerate}
Thus, $\rho$ is non-cuspidal.

\end{proof}

%\begin{eg}
%Consider the case $n=4$ and $k=2$ in the function field case. We have a split exact sequence
%\[
%1 \to \M_4(\Oh_1) \to \GL_4(\Oh_2) \overset{p}{\to} \GL_4(\Oh_1) \to 1,
%\]
%and using the representation theory for semidirect products we know that we can analyze the representations of %$\GL_4(\Oh_2)$ according to the orbits of characters of $\M_4(\Oh_1)$ and representations of their stabilizers. We %focus on the orbit of
%\[
%\theta=\left(\begin{matrix} \eta & 0 \\ 0 & \eta \end{matrix}\right), \quad \eta \in \M_2(\Oh_1),
%\]
%with $\eta$ irreducible. The centralizer in $\GL_4(\Oh_1)$ of $\theta$ is $\GL_2(\tOh_1)$. Let $\rho''$ be a cuspidal %representation of $\GL_2(\tOh_1)$, let $\rho'=p^*(\rho'')$ be its extension to ${p^{-1}(\GL_2(\tOh_1))}$, and let
%\[
%\rho=\text{Ind}_{p^{-1}(\GL_2(\tOh_1))}^{\GL_4(\Oh_2) }(\rho').
%\]
%We should show that $\rho$ which is clearly not strongly cuspidal is cuspidal.
%\end{eg}
%\bigskip

%\noindent { [Missing]}

%\bigskip

%%%%%%%%%%%%%%%%%%%%%%%%%%%%%%%%%%%%%%%%%%%%%%%%%%%%%%%%%%%%%%%%%%%%%%%%%%%%%
%         Parameterization of cuspidal representations                      %
%%%%%%%%%%%%%%%%%%%%%%%%%%%%%%%%%%%%%%%%%%%%%%%%%%%%%%%%%%%%%%%%%%%%%%%%%%%%%

\section{Construction of strongly cuspidal representations}
\label{sec:parameterization}

The construction of strongly cuspidal representations of $\GL_n(\Oh_k)$ when $k>1$ can be found, for example, in \cite{MR0233931,MR0396859,MR0492087,MR760676, MR1204652,MR1311772}. 
In this section, we recall this construction in a way that Theorems~B and C are seen to follow from it.

\subsection{Primitive characters}
\label{sec:primitive characters}
Let $E$ denote an unramified extension of $F$ of degree $n$.
Let $\OO$ be the integral closure of $\Oh$ in $E$.
The maximal ideal of $\OO$ is $\PP=\p\OO$.
Let $\OO_k=\OO/\PP^k$.
As an $\Oh_k$-module, $\OO_k$ is isomorphic to a free $\Oh_k$-module of rank $n$.
Therefore, $G_{k^n}$ can be identified with $\Aut_{\Oh_k}(\OO_k)$.
This identification is determined up to an inner automorphism of $G_{k^n}$.
Thus, the strongly cuspidal representations constructed in this section are determined up to isomorphism.

Left multiplication by elements of $\OO_k$ gives rise to $\Oh_k$-module endomorphisms of $\OO_k$.
Therefore, $\tok$ can be thought of as a subgroup of $G_{k^n}$.
Similarly, for each $r\geq k/2$, $\OO_{k-r}$ will be thought of as a subring of $\M_n(\Oh_{k-r})$.

 Strongly cuspidal representations of $G_{k^n}$ will be associated to certain characters of $\tok$ which we will call \emph{strongly primitive}.
In order to define a strongly primitive character of $\tok$ it is first necessary to define a primitive character of $\OO_1$.
\begin{defn}
  [Primitive character of $\OO_1$]
  \label{defn:primitive-field}
  A \emph{primitive character of $\OO_1$} is a homomorphism $\phi\colon\OO_1\to \C^\times$ which does not factor through any proper subfield via the trace map.
\end{defn}
The map $\OO_k\to \tok$ given by $a\mapsto 1+\p^ra$ induces an isomorphism $\OO_{k-r}\tilde\to \ker(\tok\to \OO_r^\times)$, for each $1\leq r<k$.
\begin{defn}
  [Strongly primitive character of $\tok$]
  \label{defn:primitive-tok}
  When $k>1$, a \emph{strongly primitive character of $\tok$} is a homomorphism $\omega\colon\tok\to \C^\times$ whose restriction to $\ker(\tok\to \OO_{k-1}^\times)$ is a primitive character when thought of as a character of $\OO_1$ under the above identification.
\end{defn}
The above definition does not depend on the choice of uniformizing element $\p\in \pp$.

\noindent Suppose that $r\geq k/2$.
An identification $A\mapsto \psi_A$ of $\M_n(\Oh_{k-r})$ with its Pontryagin dual was constructed in Section~\ref{sec:simil-class-assoc}.
Given $a\in \OO_{k-r}$, view it as an element of $\M_n(\Oh_{k-r})$.
Let $\phi_a$ denote the restriction of $\psi_a$ to $\OO_{k-r}$.
Then $a\mapsto \phi_a$ is an isomorphism of $\OO_{k-r}$ with its Pontryagin dual.
\subsection{Construction of strongly cuspidal representations from strongly primitive characters}
The reader may find it helpful to refer to (\ref{eq:construction}) while navigating the construction.
\label{sec:construction}
Let $l=\lceil k/2 \rceil$ be the smallest integer not less than $k/2$ and $l'=\lfloor k/2 \rfloor$ be the largest integer not greater than $k/2$.
Let $\omega$ be a strongly primitive character of $\tok$.
Let $a\in \tOh_{k-l}$ be such that the restriction of $\omega$ to $N_l\cap \tok$ (when identified with $\tOh_{k-l}$) is of the form $\phi_a$.
The strong primitivity of $\omega$ implies that the image of $a$ in $\OO_1$ does not lie in any proper subfield.
The formula
\begin{equation}
  \label{eq:sigmaomega}
  \tau_\omega(xu)=\psi_a(x)\omega(u) \text{ for all } x\in N_l \text{ and } u\in \tok,
\end{equation}
defines a homomorphism $\tau_\omega\colon N_l\tok\to \C^\times$. 
Let $L$ denote the kernel of the natural map 
$\tOh_k^{\times} \to \tOh_1^{\times}$. Then $N_lL$ is a normal subgroup of 
$N_{l'}\tok$ (note that $N_l\tok$ is not normal in $N_{l'}\tok$i, when $k$
is odd).
Let $\sigma_\omega$ denote the restriction of $\tau_\omega$ to $N_lL$. 
We have
  \begin{equation}
    \label{eq:stable}
    \sigma_\omega(yxy\inv)=\sigma_\omega(x)\text{ for all } y\in N_{l'}\tok \text{ and } x\in N_lL.
  \end{equation}
Let $q$ denote the order and $p$ denote the characteristic of $\Oh_1$.
The quotient $V=N_{l'}L/N_lL$ is naturally isomorphic to $\M_n(\Oh_1)/\OO_1$ which\footnote{Here $\M_n(\Oh_1)$ is identified with $\mathrm{End}_{\Oh_1}(\OO_1)$.}, being an abelian group where every non-trivial element has order $p$, can be viewed as a vector space over $\F_p$ of dimension $(n^2-n)\log_pq$.
Then 
  \begin{equation*}
    \beta(xN_{l}L,yN_{l}L)=\sigma_\omega([x,y]) \text{ for all } x, y\in N_{l'}L,
  \end{equation*}
  defines a non-degenerate alternating bilinear form 
$\beta\colon V\times V\to \mu_p$ , where $\mu_p$ denote the complex
$p^{\text{th}}$ roots of unity, \cite[Corollary~4.3]{MR1334228}.

The following lemma now follows from standard results on the representation
theory of finite Heisenberg groups (see e.g.,
\cite[Proposition~3]{MR0396859}).

\begin{lemma}
  \label{lemma:ltolprime}
  There exists a unique irreducible representation $\sigma'_\omega$ of $N_{l'}L$ whose restriction to $N_lL$ is $\sigma_\omega$ isotypic.
  This representation has dimension $q^{(l-l')(n^2-n)/2}$.
  Its character is given by
  \begin{equation*}
    \tr(\sigma'_\omega(x))=\begin{cases} q^{(l-l')(n^2-n)/2}\sigma_\omega(x) & \text{if } x\in N_lL,\\ 0 & \text{otherwise}.\end{cases}
  \end{equation*}
\end{lemma}

Recall from \cite[II.4, Proposition~8]{MR0354618}, that there is a unique multiplicative section $s\colon \OO_1^\times\to \tok$.
This allows us to realize $N_{l'}\tok$ as a semidirect product of $N_{l'}L$ by $\OO_1^\times$. 
Recall also, that $x\in \tok$ is called \emph{\re} if its image in $\OO_1^\times$ is not contained in any proper subfield.
\begin{lemma}
  \label{lemma:intertwiner}
  When $k$ is odd, there exists an irreducible representation $\tau'_\omega$ of $N_{l'}\tok$, which is unique up to isomorphism, whose restriction to $N_lL$ is $\sigma_\omega$-isotypic, and such that for any $x\in N_{l'}\tok$,
  \begin{equation*}
    \tr(\tau'_\omega(x)) = 
    \begin{cases}
      0 \text{ when } x \text{ is not conjugate to an element of } N_l\tok\\
      (-1)^{n-1} \omega(x) \text{ when } x\in \tok \text{ is \re.}
    \end{cases}
  \end{equation*}
\end{lemma}
\begin{proof}
  The lemma is easily deduced from \cite[Theorem~1]{MR0396859} as follows:
  the algebraic torus $T$ defined over $\Fq$ such that $T(\Fq)=\OO_1^\times$ splits over the extension $\F_{q^n}$ of $\Fq$.
  The Galois group of this extension acts on the weights of $T(\F_{q^n})$ on $V\otimes \F_{q^n}$, which simply correspond to roots of $\GL_n$.
  The Frobenius automorphism which generates this group acts as a Coxeter element on this root system.
  One may see that,  in the language of \cite[1.4.9(b)]{MR0396859}, this action has a unique symmetric orbit and $(n-2)/2$ non-symmetric orbits if $n$ is even, and no symmetric orbits and $(n-1)/2$ non-symmetric orbits if $n$ is odd.
  The symmetric orbits contribute a factor of $(-1)$ to the character values.
  The hypothesis that $u$ is not an element of any proper subfield of $\OO_1$ implies that $u$ is regular semisimple, and that no weight vanishes on it.
\end{proof}
When $k$ is even, define the representation $\tau'_\omega$ of $N_{l'}\tok$ to be just $\tau_\omega$ (see (\ref{eq:sigmaomega})).
Then, for any $k>1$, if $u\in \tok$ is an element whose image in $\OO_1^\times$ is a generates $\OO_1^\times$, we have
\begin{equation}
\label{eq:cuspidaltau}
  \tr(\tau'_\omega(u))=(-1)^{k(n-1)}\omega(u).
\end{equation}
Finally, define
\begin{equation*}
  \rho_\omega=\Ind{N_{l'}\tok}{G_{k^n}}\tau'_\omega.
\end{equation*}
This will be the strongly cuspidal representation associated to the strongly primitive character $\omega$ of $\tok$.
The representation $\rho_\omega$ is irreducible because $N_{l'}\tok$ is the centralizer of $\sigma_\omega$ in $G_{k^n}$.

The steps in the construction of $\rho_\omega$ are described schematically below for the convenience of the reader.
The diagram on the left describes the relation between the various groups involved.
The position occupied by a group in the diagram on the left is occupied by the corresponding representation that appears in the construction in the diagram on the right.

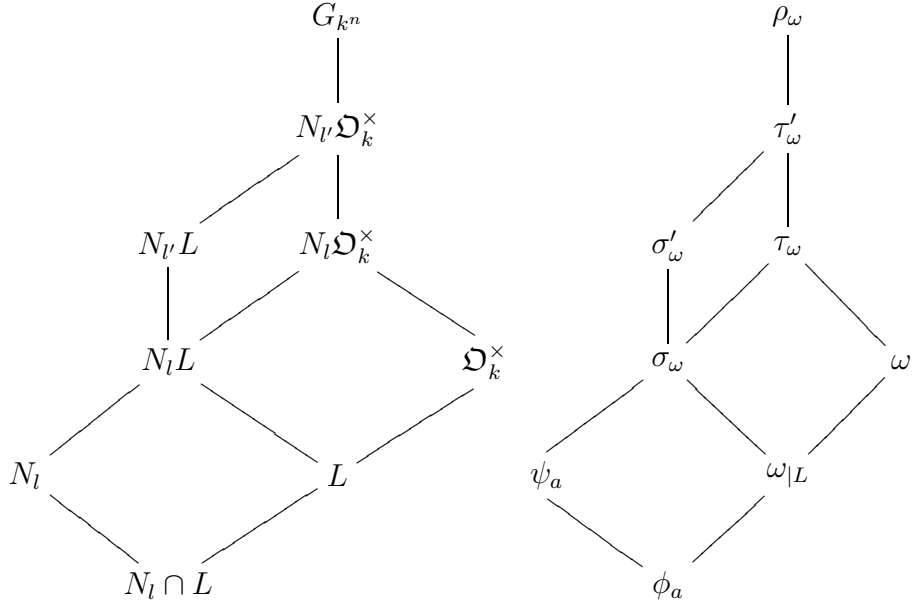
\begin{equation}
  \label{eq:construction}
  \xymatrix{ & & G_{k^n}\ar@{-}[d] &                        \qquad  & & \rho_\omega\ar@{-}[d] &  \\
    & & N_{l'}\tok\ar@{-}[dl]\ar@{-}[d] &                   \qquad & & \tau'_{\omega} \ar@{-}[dl]\ar@{-}[d] &   \\
    & N_{l'}L\ar@{-}[d]&  N_{l}\tok\ar@{-}[dr]\ar@{-}[dl] & \qquad  & \sigma'_{\omega}\ar@{-}[d]&  \tau_{\omega}\ar@{-}[dr]\ar@{-}[dl] & \\
    & N_lL \ar@{-}[dr]\ar@{-}[dl] &  & \tok\ar@{-}[dl]      \qquad  & \sigma_\omega \ar@{-}[dr]\ar@{-}[dl] &  & \omega \ar@{-}[dl]    \\
    N_l\ar@{-}[dr] & & L\ar@{-}[dl] &                       \qquad  \psi_a \ar@{-}[dr] & & \omega_{|L}\ar@{-}[dl] &       \\
    & N_l\cap L &  &                                        \qquad   &  \phi_a &  &
     }
\end{equation}
\begin{theorem}
  \label{theorem:character}
  For each strongly primitive character $\omega$ of $\tok$, $\rho_\omega$ is an irreducible representation such that
  \begin{enumerate}
  \item $\tr(\rho_\omega(g)) = 0$ if $g$ is not conjugate to an element of $N_l\tok$.
  \item if $u\in \tok$ is such that its image in $\OO_1^\times$ is not contained in any proper subfield, then
    \begin{equation*}
      \tr(\rho_\omega(u)) = (-1)^{k(n-1)}\sum_{\gamma\in \Gal(E/F)} \omega({}^\gamma u)).
    \end{equation*}
    for every $u\in \tok$, whose image in $\OO_1^\times$ lies in no proper subfield.
  \end{enumerate}
  \begin{proof}
    The first assertion follows from Lemma~\ref{lemma:ltolprime}.
    The second follows from the fact that the intersection of the conjugacy class of $u$ in $G_{k^n}$ with $\tok$ consists only of the elements ${}^\gamma u$, for $\gamma\in \Gal(E/F)$.
  \end{proof}
\end{theorem}

\subsection{The parameterization of strongly cuspidal representations of $G_{k^n}$}
\label{sec:Green}
The following is a detailed version of Theorem B.
\begin{theorem}\quad
  \begin{enumerate}
  \item For each strongly primitive character $\omega$ of $\tok$, the representation $\rho_\omega$ of $G_{k^n}$ is irreducible and strongly cuspidal.
  \item Every strongly cuspidal representation of $G_{k^n}$ is isomorphic to $\rho_\omega$ for some strongly primitive character $\omega$ of $\tok$.
  \item If $\omega'$ is another strongly primitive character of $\tok$, then $\rho_\omega$ is isomorphic to $\rho_{\omega'}$ if and only if $\omega'=\omega\circ\gamma$ for some $\gamma\in \Gal(E/F)$.
  \end{enumerate}
\end{theorem}
\begin{proof}[Proof of (1)]
  The irreducibility of $\rho_\omega$ follows from standard results on induced representations.
  To see that $\rho_\omega$ is strongly cuspidal, observe that the restriction of $\rho_\omega$ to $N_l$ contains $\psi_a$.
  This means that its restriction to $N_{k-1}$ contains $\psi_{\overline{a}}$, where $\overline a$ is the image of $a$ in $\OO_1$.
  Since this image does not lie in any proper subfield, its minimal polynomial is irreducible of degree $n$.
  Therefore, as an element of $\M_n(\Oh_1)$, its characteristic polynomial must be irreducible.
\end{proof}
\begin{proof}[Proof of (2)]
  Suppose that $\rho$ is an irreducible strongly cuspidal representation of $G_{k^n}$.
  Unwinding the definitions, one see that $\Omega_1(\rho)$ is just the image of $\Omega_{k-l}(\rho)$ under the natural map $\M_n(\Oh_{k-l})\to \M_n(\Oh_1)$.
  Let $p(t)\in \Oh_{k-l}[t]$ be the characteristic polynomial of the matrices in $\Omega_{k-l}(\rho)$.
  Denote its image in $\Oh_1[t]$ by $\overline p(t)$.
  The hypothesis on $\rho$ implies that $\overline p (t)$ is irreducible.
  Let $\tilde p(t)$ be any polynomial in $\Oh[t]$ whose image in $\Oh_{k-l}[t]$ is $p(t)$.
  By Hensel's lemma, there is a bijection between the roots of $\tilde p(t)$ in $E$ and the roots of $\overline p (t)$ in $\OO_1$.
  Consequently,
  \begin{equation*}
    \Hom_F(F[t]/\tilde p(t), E) \cong \Hom_{\Oh_1}(\Oh_1[t]/\overline p(t),\OO_1).
  \end{equation*}
  But we know that $\OO_1$ is isomorphic to $\Oh_1[t]/\overline p(t)$.
  In fact there are exactly $n$ such isomorphisms.
  Each one of these gives an embedding of $F[t]/\tilde p(t)$ in $E$.
  Since both $F[t]/\tilde p(t)$ and $E$ have degree $n$, these embeddings must be isomorphisms.
  Any root $\tilde a$ of  $\tilde p(t)$ in $E$ also lies in $\OO$.
  It is conjugate to the companion matrix of $\tilde p (t)$ in $\GL_n(\Oh)$.
  Therefore, its image $a\in \OO_{k-l}$ lies in $\Omega_{k-l}(\rho)$.
  It follows that $\rho_{|N_l}$ contains a $\psi_a$ isotypic vector.

  By applying the little groups method of Wigner and Mackey to the normal subgroup $N_l$ of $G_{k^n}$, we see that every representation of $\rho_k$ whose restriction to $N_l$ has a $\psi_a$ isotypic vector is induced from an irreducible representation of $N_{l'}\tok$ whose restriction to $N_l$ is $\psi_a$ isotypic.
  It is not difficult then to verify (by counting extensions at each stage) that the construction of $\tau'_\omega$ in Section~\ref{sec:construction} gives all such representations.
\end{proof}
\begin{proof}[Proof of (3)]
  It follows from the proof of (2) that $\tau'_{\omega_1}$ and $\tau'_{\omega_2}$ are isomorphic if and only if $\omega_1=\omega_2$.
  The Galois group $\Gal(E/F)$ acts by inner automorphisms of $G_{k^n}$ (since we have identified it with $\Aut_{\Oh_k}(\OO_k)$) preserving $N_{l'}\tok$.
  Therefore, the restriction of $\rho_{\omega_1}$ to $N_{l'}\tok$ also contains $\tau_{\omega_2}$ whenever $\omega_2$ is in the $\Gal(E/F)$-orbit of $\omega_1$, hence $\rho_{\omega_1}$ is isomorphic to $\rho_{\omega_2}$. If $\omega_1$ and $\omega_2$ do not lie in the same $\Gal(E/F)$-orbit then Theorem~\ref{theorem:character} implies that that $\rho_{\omega_1}$ can not be isomorphic to $\rho_{\omega_2}$.
\end{proof}

%%%%%%%%%%%%%%%%%%%%%%%%%%%%%%%%%%%%%%%%%%%%%%%%%%%%%%%%%%%%%%%%%%%%
%             Connection with supercuspidals
%%%%%%%%%%%%%%%%%%%%%%%%%%%%%%%%%%%%%%%%%%%%%%%%%%%%%%%%%%%%%%%%%%%%

\subsection{Connection with supercuspidal representations of $\GL_n(F)$}
In \cite[Theorem~8.4.1]{MR1204652}, Bushnell and Kutzko
proved that all the irreducible supercuspidal representations of
$\GL_n(F)$ can be obtained by compact induction from a compact subgroup
modulo the center.
One such subgroup is $F^\times\GL_n(\Oh)$.
This group is a product of $\GL_n(\Oh)$ with the infinite cyclic group
$Z_1$ generated by $\T \I$.
Thus every irreducible representation of this group is a product of a
character
of $Z_1$ with an irreducible representation of $\GL_n(\Oh)$.
An irreducible representation of $\GL_n(\Oh)$ is said to be of level $k-1$
if it
factors through $\GL_n(\Oh_k)$, but not through $\GL_n(\Oh_{k-1})$.
When $n$ is prime, the representations of $\GL_n(\Oh)$ which
give rise to supercuspidal representations are precisely those which are
of
level $k-1$, for some for $k>1$, and, when viewed as representations of
$\GL_n(\Oh_k)$, are strongly cuspidal.
For $k=1$, they are  just the cuspidal representations of $\GL_n(\Oh_1)$.
The corresponding representations of $Z\GL_n(\Oh)$ are called \emph{tr\`es
cuspidale de type $k$} by Carayol in \cite[Section~4.1]{MR760676}.
The construction that Carayol gives for these representations is the same
as
the one given here, except that the construction here is made canonical by
using
G\'erardin's results.

Let $\chi$ be any character of $Z_1$.
Set \[\pi_{\omega,\chi}:=\cInd{\GL_n(\Oh)
F^\times}{\GL_n(F)}(\rho_\omega\otimes \chi).\]
These are the supercuspidal representations of $\GL_n(F)$ associated to
$\rho_\omega$.

Let $r\colon \GL_n(\Oh)\to \GL_n(\Oh_k)$ denote the homomorphism obtained by
reduction modulo $\pp^k$.
In the notation of \cite{MR1204652}, we have
$r\inv(N_lL)=H^1(\beta,\fA)$, $r\inv(N_{l'}L)=J^1(\beta,\fA)$ and
$r\inv(N_{l'}\tok)=J(\beta,\fA)$, where $\fA=\M_n(\Oh)$ and
$\beta\in\M_n(F)$ is minimal (see \cite[(1.4.14)]{MR1204652}).
These groups are very special cases of the groups defined
in \cite[(3.1.14)]{MR1204652}.
The inflation $\eta$ of $\sigma_{\omega'}$ to $J^1(\beta,\fA)$ is a special case
of the Heisenberg representation defined in \cite[Prop.~5.1.1]{MR1204652}.

We will say that a supercuspidal representation $\pi$ of $\GL_n(F)$
belongs to the \emph{unramified series} if the field extension $F[\beta]$ of
$F$ is unramified (by \cite[(1.2.4), (6.2.3)~(i)]{MR1204652}, this is equivalent 
to saying that the $\Oh$-order $\fA$ occurring in the construction of $\pi$ is 
maximal). 
When $n$ is a prime number, Carayol has proved (see 
\cite[Theorem 8.1 (i)]{MR760676}) 
that the representations $\pi_{\omega,\chi}$ give all the supercuspidal 
representations of $\GL_n(F)$ which belong to the unramified series.
However, when $n$ is composite, the strongly cuspidal representations are not 
sufficient in order to build all the supercuspidal representations in the 
unramified series of $\GL_n(F)$ (see for instance Howe's construction in 
\cite{MR0492087}). 
Since all the supercuspidal representations of $\GL_n(F)$ are known \cite{MR1204652},
it would be natural to try restricting them to $\GL_n(\Oh)$ and see if one
get cuspidal representations among the components. On the other hand we 
observe that our notion of cuspidality is in a sense stronger
than the usual notion of supercuspidality for representations of
$\GL_n(F)$, since supercuspidality can only see geometric induction.

%%%%%%%%%%%%%%%%%%%%%%%%%%%%%%%%%%%%%%%%%%%%%%%%%%%%%%%%%%%%%%%%%%%%%%%%%%%%%
%         Complexity of the classification problem                          %
%%%%%%%%%%%%%%%%%%%%%%%%%%%%%%%%%%%%%%%%%%%%%%%%%%%%%%%%%%%%%%%%%%%%%%%%%%%%%

\section{Complexity of the classification problem}
\label{sec:complexity}
  In this section it will be shown that the representation theory of the family of groups $G_{k^n}$ actually involves the much larger family, $G_{\lambda,E}$ ($\lambda \in \Lambda$, $E/F$ unramified), which was defined in Section~\ref{sec:some-groups}, even when $k=2$.
\begin{theorem}
  \label{theorem:sensational}
  Let $F=\Fq((\p))$ be a local function field.
  Then the problems of constructing all the irreducible representations of the following groups are equivalent:
  \begin{enumerate}
  \item $G_{2^n,F}$ for all $n\in \N$.
  \item $G_{k^n,F}$ for all $k,n\in \N$.
  \item $G_{\lambda,E}$ for all partitions $\lambda$ and all unramified extensions $E$ of $F$.
  \end{enumerate}
\end{theorem}
\begin{proof}
  Obviously (3) implies (2), which implies (1).
  That (1) implies (3) follows from the somewhat more precise formulation in Theorem~\ref{theorem:complexity}.
\end{proof}
\begin{theorem}
  \label{theorem:complexity}
  Let $F$ be a local function field.
  Then the problem of constructing all the irreducible representations of $G_{2^n,F}$ is equivalent to the problem of constructing all the irreducible representations of all the groups $G_{\lambda,E}$, where $E$ ranges over all unramified extensions of $F$ of degree $d$ and $\lambda$ ranges over all partitions such that $d(\lambda_1r_1+\cdots+\lambda_lr_l)\leq n$.
\end{theorem}
\begin{proof}
  When $F$ is a local function field, $G_{2^n}$ is isomorphic to the semidirect product of $\GL_n(\Oh_1)$ by $M_n(\Oh_1)$.
  The \emph{little groups method} of Wigner and Mackey (see e.g., \cite[Prop.~25]{MR0450380}) shows that constructing the irreducible representations of $G_{2^n}$ is equivalent to constructing the irreducible representations of the centralizers in $\GL_n(\Oh_1)$ of all the multiplicative characters of $M_n(\Oh_1)$.
  Pick any $\chi$ for which the space $V_\chi$ of $\chi$-isotypic vectors is non-zero.
  By the discussion in Section~\ref{sec:simil-class-assoc}, these subgroups of $\GL_n(\Oh_1)$ are the same as the centralizer groups of matrices.
We will see below that these centralizer groups are products of groups of the form $G_{\lambda,E}$ that appear in the statement of Theorem~\ref{theorem:complexity}.

  Let $A\in \M_n(\Oh_1)$.
  Then, $\Oh_1^n$ can be thought of as a $\Oh_1[\T]$-module where $\T$ acts through $A$.
  The centralizer of $A$ is the automorphism group of this $\Oh_1[\T]$-module.
  For each irreducible monic polynomial $f(\T)\in\Oh_1[\T]$  of degree $d$ which divides the characteristic polynomial of $A$, the $f$-primary part of this module is isomorphic to
  \begin{equation*}
    (\Oh_1[\T]/f(\T)^{\lambda_1})^{r_1}\oplus\cdots \oplus(\Oh_1[\T]/f(\T)^{\lambda_l})^{r_l},
  \end{equation*}
  for some partition $\lambda$.
  \begin{lemma}
    \label{lemma:Hensel}
    Let $\OO_1=\Oh_1[\T]/f(\T)$.
    The rings $\Oh_1[\T]/f(\T)^k$ and $\OO_1[u]/u^k$ are isomorphic for every $k>0$.
  \end{lemma}
  \begin{proof}
    It will be shown by induction that there exists a sequence $\{q_k(\T)\}$, in $\Oh_1[\T]$ such that
    \begin{enumerate}
    \item $q_1(\T)=\T$,
    \item $q_{k+1}(\T)\equiv q_k(\T)\mod f(\T)^k$ for all $k>0$, and,
    \item $f(q_k(\T))\in f(\T)^k$ for all $k>0$.
    \end{enumerate}
    For $k=1$ the result is obvious.
    Suppose that $q_k(\T)$ has been constructed.
    Since $\Oh_1$ is a perfect field and $f(\T)$ is irreducible, $f'(\T)$ is not identically $0$.
    It follows that $f'(\T)$ does not divide $f(\T)$.
    Since $q_k(\T)\equiv \T \mod f(\T)$, $f'(q_k(\T))$ does not divide $f(\T)$.
    Therefore, the congruence
    \begin{equation*}
      f(q_k(\T))+f(\T)^k h(\T) f'(q_k(\T)) \equiv 0 \mod f(\T)^{k+1}
    \end{equation*}
    can be solved for $h(\T)$.
    Let $h_0(\T)$ be a solution.
    Take $q_{k+1}(\T)=q_k(\T)+f(\T)^kh_0(\T)$.
    The sequence $\{q_k(\T)\}$ constructed in this manner has the required properties.

    Now note that $\OO_1[u]/u^k \cong \Oh_1[\T,u]/(f(\T),u^k)$.
    One may define a ring homomorphism
    \begin{equation*}
      \Oh_1[\T,u]/(f(\T),u^k)\to \Oh_1[\T]/f(\T)^k
    \end{equation*}
    by $\T\mapsto q_k(\T)$ and $u\mapsto f(\T)$.
    Since $q_k(\T)\equiv \T \mod f(\T)$, $\T$ lies in the image of this map, so it is surjective.
    As vector spaces over $\Oh_1$ both rings have dimension $kd$.
    Therefore, it is an isomorphism.
  \end{proof}
  It follows from Lemma~\ref{lemma:Hensel} that the automorphism group of the $f$-primary part of $\Oh_1^n$ is $G_{\lambda,E}$, where $E$ is an unramified extension of $F$ of degree $d$.
  The automorphism group of the $\Oh_1[\T]$-module $\Oh_1^n$ is the product of the automorphism groups of its $f$-primary parts.
  Therefore, the centralizer of $A$ in $G_{1^n}$ is a product of groups of the form $G_{\lambda,E}$.
  Considerations of dimension show that $d(\lambda_1r_1+\cdots+\lambda_lr_l)\leq n$ for each $G_{\lambda,E}$ that occurs.

  Conversely given $\lambda$ and $d$ satisfying the above inequality, take an irreducible polynomial $f(\T)\in \Oh_1[\T]$ of degree $d$.
  Define
  \begin{equation*}
    J_k(f)=\begin{pmatrix} C_f & 0 & 0 & \cdots & 0 & 0\\
    \I_d & C_f & 0 & \cdots & 0 & 0\\
    0 & \I_d & C_f & \cdots & 0 & 0 \\
    \vdots & \vdots & \vdots & \ddots & \vdots & \vdots\\
    0 & 0 & 0 & \cdots & C_f & 0\\
    0 & 0 & 0 & \cdots & \I_d & C_f
    \end{pmatrix}_{kd\times kd},
  \end{equation*}
  where $C_f$ is any matrix with characteristic polynomial $f$.
  Let
  \begin{equation*}
  A=J_{\lambda_1}(f)^{\oplus r_1}\oplus \cdots \oplus J_{\lambda_l}(f)^{\oplus r_l}\oplus J_k(\T-a),
  \end{equation*}
 where $a\in \Oh_1$ is chosen so that $\T-a\neq f(\T)$ and $k=n-d(\lambda_1r_1+\cdots+\lambda_lr_l)$.
The centralizer of $A$ contains $G_{\lambda,E}$ as a factor.
\end{proof}

%%%%%%%%%%%%%%%%%%%%%%%%%%%%%%%%%%%%%%%%%%%%%%%%%%%%%%%%%%%%%%%%%%%%%%%%%%%%%
%         Zeta function                                                     %
%%%%%%%%%%%%%%%%%%%%%%%%%%%%%%%%%%%%%%%%%%%%%%%%%%%%%%%%%%%%%%%%%%%%%%%%%%%%%

\section{The zeta function associated to $G_{k^n}$}
\label{sec:zeta}
 In \cite{MR0396780}, Springer attaches a zeta function to irreducible representations of $\GL_n(\Oh_1)$, and proves that for cuspidal representations it satisfies a functional equation. Later on, Macdonald \cite{MI4} shows that a functional equation holds for any irreducible representation, provided that it has no $1$-component, namely, it is not contained in $i_{n-1,1}(\rho,1)$ for any representation $\rho$ of $\GL_{n-1}(\Oh_1)$. Moreover, Macdonald establishes a bijection between irreducible representations of $\GL_n(\Oh_1)$, and equivalence classes of tamely ramified representations of the Weil-Deligne group $W'_F$, which preserves certain $L$ and $\varepsilon$ factors.

In this section we attach a zeta function to any irreducible representation of $G_{k^n}$ and show that it satisfies a functional equation, provided that $\rho$ does not lie in the infinitesimally induced series. We follow closely \cite{MI4} and make the necessary adaptations.

\medskip

The map $F\to \C^\times$ given by $x\mapsto \psi(\pi^kx)$, when restricted to $\Oh$, factors through an additive character $\psi_k$ of $\Oh_k$, which does not factor through $\Oh_{k-1}$. Denote $G=G_{k^n}$ and $M=\M_{k^n}=\M_n(\Oh_{k})$, and let $\C(M)$ denote complex valued functions on $M$. For $f \in \C(M)$ define its Fourier transform by
\[
\hat{f}(x)=|M|^{-1/2}\sum_{y \in M}f(y)\psi_k\left(\tr(xy)\right),
\]
so that $\hat{\hat{f}}(x)=f(-x)$. Let $(\rho,V)$ be a finite dimensional representation of $G$. For each $f \in \C(M)$ define the zeta-function
\[
\mathcal{Z}(f,\rho)=\sum_{g\in G}f(g)\rho(g) \in \text{End}_{\C}(V).
\]
Also, for $x \in M$ let
\[
\mathcal{W}(\rho,\psi;x)=|M|^{-1/2}\sum_{g \in G}\psi_k(\tr(gx))\rho(g).
\]
The following lemma is straightforward.
\begin{lemma}\label{lem} \qquad
\begin{enumerate}
\item [(a)] $\mathcal{Z}(f,\rho)=\sum_{x \in M}\hat{f}(-x)\mathcal{W}(\rho,\psi;x)$.

\item [(b)] $\mathcal{W}(\rho,\psi;xg)=\rho(g)^{-1}\mathcal{W}(\rho,\psi;x)$.

\item [(c)] $\mathcal{W}(\rho,\psi;gx)=\mathcal{W}(\rho,\psi;x)\rho(g)^{-1}$.

\end{enumerate}
\end{lemma}
In particular, setting $x=1$ in parts (b)-(c) of Lemma \ref{lem} shows that $\mathcal W(\rho,\psi;1)$ commutes with $\rho(g)$ for all $g \in G$. Therefore, if $\rho$ is irreducible, then $\mathcal{W}(\rho,\psi;1)$ is a scalar multiple of $\rho(1)$. Following \cite{MI4} we write $\varepsilon(\rho,\psi)\rho(1)=\mathcal{W}(\check{\rho},\psi;1)$, where $\check\rho$ is the contragredient of $\rho$, i.e. $\check{\rho}(g)={^t\rho(g^{-1})}$ and $\varepsilon(\rho,\psi)$ is a complex number.
\begin{prop}\label{vanishing} Let $\rho$ be an irreducible representation of $G$ which does not lie in the infinitesimally induced series. Then $\mathcal{W}(\rho,\psi;x)=0$ for all $x \in M \smallsetminus G$.
\end{prop}
\begin{proof} Let $H_x=\{ g \in G ~|~ gx=x \}$. For $g \in H_x$ we have
\[
\mathcal{W}(\rho,\psi;x)=\mathcal{W}(\rho,\psi;gx)=\mathcal{W}(\rho,\psi;x)\rho(g^{-1})=\mathcal{W}(\rho,\psi;x)\rho(e_{H_x}),
\]
where $\rho(e_{H_x})=|H_x|^{-1} \sum_{g\in H_x}\rho(g)$. Hence, it suffices to show that $\rho(e_{H_x})=0$ for $x \in M \smallsetminus G$. Since $\rho(e_{H_x})$ is the idempotent projecting $V$ onto $V^{H_x}$, it is enough to to show that the latter subspace is null. Let $\mu=(\mu_1,\ldots,\mu_n)$ be the divisor type of $x$. Namely, $0 \le \mu_1 \le \cdots \le \mu_n \le k$, such that acting with $G$ on the right and on the left gives: $gxh=d_{\mu}=\text{diag}(\varpi^{\mu_1},\ldots,\varpi^{\mu_n})$. Then $H_x=gH_{d_{\mu}}g^{-1}$. Now for any $\mu$ we have $H_{d_{\mu}} \supset H_{d_{\nu}}$, where $\nu=(0,0,\ldots,0,1)$. Therefore, it is enough to show that $V^{H_{d_\nu}}=(0)$. The subgroup $H_{d_{\nu}}$ is given explicitly by
\[
H_{d_{\nu}}=\left[\begin{matrix} \I_{n-1} &  \varpi^{k-1}\star \\ 0 &  1+\varpi^{k-1}\star \end{matrix}\right]=U_{(k^{n-1},k-1)\hookrightarrow k^n} \text{ (see Section \ref{subsec:infinitesimal})}.
\]
It follows that $V^{H_{d_\nu}}=(0)$ if $\rho$ does not lie in the infinitesimally induced series.
\end{proof}

\begin{theorem}\label{theorem:zeta} For all $f \in \C(M)$ and all irreducible representations $\rho$ of $G$ that do not lie in the infinitesimally induced series, we have
\[
{^t\mathcal{Z}}(\hat{f},\check{\rho})=\varepsilon(\rho,\psi) \mathcal{Z}(f,\rho).
\]

\end{theorem}

\begin{proof} If $\rho$ does not lie in the infinitesimally induced series then nor does $\check{\rho}$, and hence $\mathcal{W}(\check{\rho},\psi;x)=0$ for all $x \in M \smallsetminus G$. We get
\[
\begin{split}
{^t\mathcal{Z}}(\hat{f},\check{\rho}) &= \sum_{g \in G} {\hat{\hat{f}}(-g) {^t\mathcal{W}}(\check{\rho},\psi;g)} \qquad \qquad \qquad \qquad  \qquad \qquad \text{(by Lemma \ref{lem}(a))} \\
&=\mathcal{W}(\check{\rho},\psi;1) \sum_{g \in G}f(g)\rho(g) =\varepsilon(\rho,\psi) \mathcal{Z}(f,\rho) \qquad \qquad \text{(by Lemma \ref{lem}(c))}.
\end{split}
\]
\end{proof}

The possibility of relating representations of $G_{k^n}$ with some equivalence classes of representations of the Weil-Deligne group $W'_F$, and consequently extending Macdonald correspondence to higher level, seems very appealing. However, such correspondence, if exists, is expected to be much more involved in view of the complexity of the representation theory of $G_{k^n}$.
\section{Cuspidal representations which are not strongly cuspidal}
\label{sec:alex}
In this section we give a description of all the cuspidal representations
of $G_{2^{4}}\cong\text{GL}_{4}(\mathfrak{o}_{2})$ in the sense of
Definition 4.2. This shows in particular the existence of representations
which are cuspidal, yet are not strongly cuspidal.

Let $\lambda=(2^{4})$ and put $G=G_{\lambda}$. If $\pi$ is a cuspidal
representation of $G$, then by Proposition 4.4 it is primary, that
is, its orbit in $M_{4}(\mathbb{F}_{q})$ consists of matrices whose
characteristic polynomial is of the form $f(X)^{a}$, where $f(X)$
is an irreducible polynomial. If $a=1$, then $\pi$ is strongly cuspidal
(by definition), and such representations were described in Section~5.
On the other hand, $f(X)$ cannot have degree 1, because then it would
be infinitesimally induced from $G_{(2^{3},1)}$, up to $1$-dimensional
twist (cf. the end of the proof of Theorem 4.3). We are thus reduced
to considering representations whose characteristic polynomial is
a reducible power of a non-linear irreducible polynomial. In the situation
we are considering, there is only one such possibility, namely the
case where $f(X)$ is quadratic, and $a=2$. Let $\eta$ denote an
element which generates the extension $\mathbb{F}_{q^{2}}/\mathbb{F}_{q}$.
We consider $M_{2}(\mathbb{F}_{q^{2}})$ as embedded in $M_{4}(\mathbb{F}_{q})$
via the embedding $\mathbb{F}_{q^{2}}\hookrightarrow M_{2}(\mathbb{F}_{q})$,
by choosing the basis $\{1,\eta\}$ for $\mathbb{F}_{q^{2}}$ over
$\mathbb{F}_{q}$. Rational canonical form implies that in $M_{4}(\mathbb{F}_{q})$
there are two conjugation orbits containing elements with two equal
irreducible $2\times2$ blocks on the diagonal, one regular, and one
which is not regular (we shall call the latter \emph{irregular}),
represented by the following elements, respectively:\[
\beta_{1}=\begin{pmatrix}\eta & 1\\
0 & \eta\end{pmatrix},\qquad\beta_{2}=\begin{pmatrix}\eta & 0\\
0 & \eta\end{pmatrix}.\]
Therefore, any irreducible cuspidal non-strongly cuspidal representation
of $G$ has exactly one of the elements $\beta_{1}$ or $\beta_{2}$
in its orbit.

Denote by $N_{1}\cong1+\varpi M_{2}(\mathfrak{o}_{2})$ the kernel
of the reduction map $G=G_{2^{4}}\rightarrow G_{1^{4}}$. As in Section
2.2, let $\psi$ be a fixed non-trivial additive character of $F$,
trivial on $\mathfrak{o}$. Then for each $\beta\in M_{4}(\mathbb{F}_{q})$
we have a character $\psi_{\beta}:N_{1}\rightarrow\mathbb{C}^{\times}$
defined by \[
\psi_{\beta}(1+\varpi x)=\psi(\mathrm{\varpi^{-1}Tr}(\beta x)).\]

The group $G$ acts on its normal subgroup $N_{1}$ via conjugation,
and thus on the set of characters of $N_{1}$ via the {}``coadjoint
action''. For any character $\psi_{\beta}$ of $N_{1}$, we write
\[
G(\psi_{\beta}):=\Stab_{G}(\psi_{\beta}).\]
By Proposition 2.3 in \cite{MR1334228}, the stabilizer $G(\psi_{\beta})$
is the preimage in $G$ of the centralizer $C_{G_{1^{4}}}(\beta)$,
under the reduction mod $\mathfrak{p}$ map. 

By definition, an irreducible representation $\pi$ of $G$ is cuspidal
iff none of its 1-dimensional twists $\pi\otimes\chi\circ\det$ has
any non-zero vectors fixed under any group $U_{i,j}$ or $U_{\lambda\hookrightarrow2^{4}}$,
or equivalently (by Frobenius reciprocity), if $\pi\otimes\chi\circ\det$
does not contain the trivial representation $\mathbf{1}$ when restricted
to $U_{i,j}$ or $U_{\lambda\hookrightarrow2^{4}}$. The groups $U_{i,j}$
are analogs of unipotent radicals of (proper) maximal parabolic subgroups
of $G$, and $U_{\lambda\hookrightarrow2^{4}}$ are the infinitesimal
analogs of unipotent radicals (cf. Section~3). Note that since $\Ind{U_{i,j}}{G}\mathbf{1}=\Ind{U_{i,j}}{G}(\mathbf{1}\otimes\chi\circ\det)=(\Ind{U_{i,j}}{G}\mathbf{1})\otimes\chi\circ\det$,
for any character $\chi:\mathfrak{o}_{2}^{\times}\rightarrow\mathbb{C}^{\times}$,
a representation is a subrepresentation of a geometrically induced
representation if and only if all its one-dimensional twists are.

In our situation, that is, for $n=4$ and $k=2$, there are three
distinct geometric stabilizers, $P_{1,3}$, $P_{2,2}$, and $P_{3,1}$
with {}``unipotent radicals'' $U_{1,3}$, $U_{2,2}$, and $U_{3,1}$,
respectively. Thus a representation is a subrepresentation of a geometrically
induced representation if and only if it is a component of $\Ind{U_{i,j}}{G}\mathbf{1}$,
for some $(i,j)\in\{(1,3),(2,2),(3,1)\}$. Furthermore, there are
three partitions, written in descending order, which embed in $2^{4}$
and give rise to non-trivial infinitesimal induction functors, namely\[
(2,1^{3}),\ (2^{2},1^{2}),\ (2^{3},1).\]
Thus a representation is a subrepresentation of an infinitesimally
induced representation if and only if it is a component of $\Ind{U_{\lambda\hookrightarrow2^{4}}}{G}\mathbf{1}$,
for some partition $\lambda$ as above. Because of the inclusions\[
U_{(2,1^{3})\hookrightarrow2^{4}}\subset U_{(2^{2},1^{2})\hookrightarrow2^{4}}\subset U_{(2^{3},1)\hookrightarrow2^{4}},\]
an irreducible representation of $G$ is a component of an infinitesimally
induced representation if and only if it is a component of $\Ind{U_{(2,1^{3})\hookrightarrow2^{4}}}{G}\mathbf{1}$.

\begin{lem}
\label{lem:first}Suppose that $\pi$ is an irreducible representation
of $G$ whose orbit contains either $\beta_{1}$ or $\beta_{2}$.
Then $\pi$ is not an irreducible component of any representation
geometrically induced from $P_{1,3}$ or $P_{3,1}$. Moreover, no
$1$-dimensional twist of $\pi$ is an irreducible component of an
infinitesimally induced representation.
\end{lem}
\begin{proof}
If $\pi$ were a component of $\Ind{U_{1,3}}{G}\mathbf{1}$, then
$\langle\pi|_{U_{1,3}},\mathbf{1}\rangle\neq0$, so in particular
$\langle\pi|_{N_{1}\cap U_{1,3}},\mathbf{1}\rangle\neq0$, which implies
that $\pi|_{N_{1}}$ contains a character $\psi_{b}$, where $b=(b_{ij})$
is a matrix such that $b_{i1}=0$ for $i=2,3,4$. This means that
the characteristic polynomial of $b$ would have a linear factor,
which contradicts the hypothesis. The case of $U_{3,1}$ is treated
in exactly the same way, except that the matrix $b$ will have $b_{4j}=0$
for $j=1,2,3$. The case of infinitesimal induction is treated using
the same kind of argument. Namely, if $\pi$ were a component of $\Ind{U_{(2,1^{3})\hookrightarrow2^{4}}}{G}\mathbf{1}$,
then $U_{(2,1^{3})\hookrightarrow2^{4}}\subset N_{1}$ and $\langle\pi|_{U_{(2,1^{3})\hookrightarrow2^{4}}},\mathbf{1}\rangle\neq0$,
which implies that $\pi|_{N_{1}}$ contains a character $\psi_{b}$,
where $b=(b_{ij})$ is a matrix such that $b_{1j}=0$ for $j=1,\dots,4$.
A $1$-dimensional twist of $\pi$ would then contain a character
$\psi_{aI+b}$, where $a$ is a scalar and $I$ is the identity matrix.
The matrix $aI+b$ has a linear factor in its characteristic polynomial,
which contradicts the hypothesis.
\end{proof}
We now consider in order representations whose orbits contain $\beta_{1}$
or $\beta_{2}$, respectively. In the following we will write $\bar{P}_{2,2}$
and $\bar{U}_{2,2}$ for the images mod $\mathfrak{p}$ of the groups
$P_{2,2}$ and $U_{2,2}$, respectively.

\subsection{The regular cuspidal representations}

Assume that $\pi$ is an irreducible representation of $G$ whose
orbit contains $\beta_{1}$. Since $\beta_{1}$ is a regular element,
the representation $\pi$ can be constructed explicitly as an induced
representation (cf. \cite{MR1334228}). In particular, it is shown
in \cite{MR1334228} that there exists a $1$-dimensional representation
$\rho$ of $G(\psi_{\beta_{1}})$ (uniquely determined by $\pi$)
such that $\rho|_{N_{1}}=\psi_{\beta_{1}}$, and such that\[
\pi=\Ind{G(\psi_{\beta_{1}})}{G}\rho.\]

\begin{prop}
\label{pro:cuspidal beta1}The representation $\pi$ is cuspidal if
and only if $\rho$ does not contain the trivial representation of
$G(\psi_{\beta_{1}})\cap U_{2,2}$.
\end{prop}
\begin{proof}
Lemma \ref{lem:first} shows that $\pi$ is cuspidal if and only if
it is not a component of $\Ind{U_{2,2}}{G}\mathbf{1}$. By Mackey's
intertwining number theorem (cf. \cite{Curtis_Reiner}, 44.5), we
have\[
\langle\pi,\Ind{U_{2,2}}{G}\mathbf{1}\rangle=\langle\Ind{G(\psi_{\beta_{1}})}{G}\rho,\Ind{U_{2,2}}{G}\mathbf{1}\rangle=\sum_{x\in G(\psi_{\beta_{1}})\backslash G/U_{2,2}}\langle\rho|_{G(\psi_{\beta_{1}})\cap\leftexp{x}{U_{2,2}}},\mathbf{1}\rangle,\]
so this number is zero if and only if $\langle\rho|_{G(\psi_{\beta_{1}})\cap\leftexp{x}{U_{2,2}}},\mathbf{1}\rangle=0$
for each $x\in G$. Assume that $\pi$ is cuspidal. Then in particular,
taking $x=1$, we have $\langle\rho|_{G(\psi_{\beta_{1}})\cap U_{2,2}},\mathbf{1}\rangle=0$.

Conversely, assume that $\pi$ is not cuspidal. Then $\langle\rho|_{G(\psi_{\beta_{1}})\cap\leftexp{x}{U_{2,2}}},\mathbf{1}\rangle\neq0$,
for some $x\in G$, and in particular, $\langle\rho|_{N_{1}\cap\leftexp{x}{U_{2,2}}},\mathbf{\mathbf{1}}\rangle=\langle\psi_{\beta_{1}}|_{N_{1}\cap\leftexp{x}{U_{2,2}}},\mathbf{1}\rangle\neq0$.
Write $\bar{x}$ for $x$ modulo $\mathfrak{p}$. Now $\psi_{\beta_{1}}|_{N_{1}\cap\leftexp{x}{U_{2,2}}}=\psi_{\beta_{1}}|_{\leftexp{x}{(N_{1}\cap U_{2,2})}}$,
and $\psi_{\beta_{1}}(\leftexp{x}{g})=\psi_{\bar{x}^{-1}\beta_{1}\bar{x}}(g)$,
for any $g\in N_{1}\cap U_{2,2}$. Let $\bar{x}^{-1}\beta_{1}\bar{x}$
be represented by the matrix\[
\begin{pmatrix}A_{11} & A_{12}\\
A_{21} & A_{22}\end{pmatrix},\]
where each $A_{ij}$ is a $2\times2$-block. Then from the definition
of $\psi_{\bar{x}^{-1}\beta_{1}\bar{x}}$ and the condition $\psi_{\bar{x}^{-1}\beta_{1}\bar{x}}(g)=1$,
for all $g\in N_{1}\cap U_{2,2}$, it follows that $A_{21}=0$; thus
\[
\bar{x}^{-1}\beta_{1}\bar{x}\in\bar{P}_{2,2}.\]
Since $\bar{x}^{-1}\beta_{1}\bar{x}$ is a block upper-triangular
matrix with the same characteristic polynomial as $\beta_{1}$, we
must have $A_{11}=B_{1}\eta B_{1}^{-1}$, $A_{22}=B_{2}\eta B_{2}^{-1}$,
for some $B_{1},B_{2}\in\mbox{GL}_{2}(\mathbb{F}_{q})$. Then there
exists $p\in\bar{P}_{2,2}$ such that \[
(\bar{x}p)^{-1}\beta_{1}(\bar{x}p)=\begin{pmatrix}\eta & B\\
0 & \eta\end{pmatrix},\]
for some $B\in M_{2}(\mathbb{F}_{q})$ (in fact, we can take $p=\left(\begin{smallmatrix}B_{1}^{-1} & 0\\
0 & B_{2}^{-1}\end{smallmatrix}\right)$). The Levi decomposition%
%\begin{comment}
%We can conjugate the Levis subgroup by some element in the unip. rad
%to get another Levi decomp.
%\end{comment}
{} $\bar{P}_{2,2}=\left(\begin{smallmatrix}* & 0\\
0 & *\end{smallmatrix}\right)\left(\begin{smallmatrix}1 & *\\
0 & 1\end{smallmatrix}\right)$ (written in block matrix form) applied to $\beta_{1}$ and $(\bar{x}p)^{-1}\beta_{1}(\bar{x}p)$
implies that the semisimple parts $(\bar{x}p)^{-1}\left(\begin{smallmatrix}\eta & 0\\
0 & \eta\end{smallmatrix}\right)(\bar{x}p)$ and $\left(\begin{smallmatrix}\eta & 0\\
0 & \eta\end{smallmatrix}\right)$ are equal, that is, \[
\bar{x}p\in C_{G_{1^{4}}}(\left(\begin{smallmatrix}\eta & 0\\
0 & \eta\end{smallmatrix}\right))=G(\beta_{2})\cong\mbox{GL}_{2}(\mathbb{F}_{q^{2}}).\]
Now, in $G(\beta_{2})$, the equation $(\bar{x}p)^{-1}\beta_{1}(\bar{x}p)=\left(\begin{smallmatrix}\eta & B\\
0 & \eta\end{smallmatrix}\right)$ implies that $\bar{x}p\in\left(\begin{smallmatrix}* & *\\
0 & *\end{smallmatrix}\right)\cap G(\beta_{2})\subset\bar{P}_{2,2}$, so $\bar{x}\in\bar{P}_{2,2}$, and hence $x\in N_{1}P_{2,2}$. The
facts that $U_{2,2}$ is normal in $P_{2,2}$, and that $\langle\rho|_{G(\psi_{\beta_{1}})\cap\leftexp{x}{U_{2,2}}},\mathbf{1}\rangle$
only depends on the right coset of $x$ modulo $N_{1}$ then imply
that \[
0\neq\langle\rho|_{G(\psi_{\beta_{1}})\cap\leftexp{x}{U_{2,2}}},\mathbf{1}\rangle=\langle\rho|_{G(\psi_{\beta_{1}})\cap U_{2,2}},\mathbf{1}\rangle.\]

\end{proof}
The preceding proposition shows that we can construct all the cuspidal
representations of $G$ with orbit containing $\beta_{1}$ by constructing
the corresponding $\rho$ on $G(\psi_{\beta_{1}})$. Since $\psi_{\beta_{1}}$
is trivial on $N_{1}\cap U_{2,2}$, we can extend $\psi_{\beta_{1}}$
to a representation of $(G(\psi_{\beta_{1}})\cap U_{2,2})N_{1}$,
trivial on $G(\psi_{\beta_{1}})\cap U_{2,2}$. Then $\psi_{\beta_{1}}$
can be extended to a representation $\tilde{\psi}_{\beta_{1}}$ on
the whole of $G(\psi_{\beta_{1}})$, such that $\tilde{\psi}_{\beta_{1}}$
is trivial on $G(\psi_{\beta_{1}})\cap U_{2,2}$ (this incidentally
shows that there exist irreducible non-cuspidal representations of
$G$ whose orbit contains $\beta_{1}$). Now let $\theta$ be a representation
of $G(\psi_{\beta_{1}})$ obtained by pulling back a representation
of $G(\psi_{\beta_{1}})/N_{1}$ that is non-trivial on $(G(\psi_{\beta_{1}})\cap U_{2,2})N_{1}/N_{1}$.
Then $\rho:=\theta\otimes\tilde{\psi}_{\beta_{1}}$ is a representation
of $G(\psi_{\beta_{1}})$ which is a lift of $\psi_{\beta_{1}}$,
and which is non-trivial on $G(\psi_{\beta_{1}})\cap U_{2,2}$. By
a standard fact in representation theory, all the lifts of $\psi_{\beta_{1}}$
to $G(\psi_{\beta_{1}})$ are of the form $\theta\otimes\tilde{\psi}_{\beta_{1}}$
for some $\theta$ trivial on $N_{1}$. Thus all the representations
of $G(\psi_{\beta_{1}})$ which are lifts of $\psi_{\beta_{1}}$ and
which are non-trivial on $G(\psi_{\beta_{1}})\cap U_{2,2}$, are of
the form above, namely $\theta\otimes\tilde{\psi}_{\beta_{1}}$ where
$\theta$ is trivial on $N_{1}$ but non-trivial on $G(\psi_{\beta_{1}})\cap U_{2,2}$.
We note that in the regular case, distinct representations $\theta$
give rise to distinct lifts $\theta\otimes\tilde{\psi}_{\beta_{1}}$.
This can be seen by a counting argument, in the following way. Because
$\beta_{1}$ lies in a regular orbit, we can write $G(\psi_{\beta_{1}})=C_{G}(\hat{\beta}_{1})N_{1}$,
for some element $\hat{\beta}_{1}\in M_{2}(\mathfrak{o}_{2})$ with
image $\beta_{1}$ mod $\mathfrak{p}$. Then because $C_{G}(\hat{\beta}_{1})$
is abelian, there are exactly $(C_{G}(\hat{\beta}_{1}):C_{G}(\hat{\beta}_{1})\cap N_{1})=|G(\psi_{\beta_{1}})/N_{1}|$
characters $\chi$ of $C_{G}(\hat{\beta}_{1})$ which agree with $\psi_{\beta_{1}}$
on $C_{G}(\hat{\beta}_{1})\cap N_{1}$, and each of them gives rise
to a representation $\chi\psi_{\beta_{1}}$ of $G(\psi_{\beta_{1}})$
defined by $\chi\psi_{\beta_{1}}(cn)=\chi(c)\psi_{\beta_{1}}(n)$,
for $c\in C_{G}(\hat{\beta}_{1})$, $n\in N_{1}$. Clearly every lift
of $\psi_{\beta_{1}}$ to $G(\psi_{\beta_{1}})$ must be equal to
some such $\chi$ on $C_{G}(\hat{\beta}_{1})$, and distinct $\chi$
give rise to distinct representations $\chi\psi_{\beta_{1}}$. Since
the number of lifts of $\psi_{\beta_{1}}$ to $G(\psi_{\beta_{1}})$
is thus equal to the number of representations of $G(\psi_{\beta_{1}})/N_{1}$,
we see that distinct $\theta$ give rise to distinct representations
$\theta\otimes\tilde{\psi}_{\beta_{1}}$. Now by a standard result
in Clifford theory, distinct irreducible representations of $G(\psi_{\beta_{1}})$
containing $\psi_{\beta_{1}}$ (when restricted to $N_{1}$) induce
to distinct irreducible representations of $G$. Thus, distinct representations
$\theta$ give rise to distinct representations $\Ind{G(\psi_{\beta_{1}})}{G}\rho$,
although the correspondence $\theta\mapsto\Ind{G(\psi_{\beta_{1}})}{G}\rho$
is by no means canonical, due to the choice of $\tilde{\psi}_{\beta_{1}}$.
Similarly, if we are considering the lifts $\chi\psi_{\beta_{1}}$,
then the construction depends on the choice of $\hat{\beta}_{1}$.

The above parameterizations of representations of $G(\psi_{\beta_{1}})$
containing $\psi_{\beta_{1}}$, both involve non-canonical choices,
although the set of representations obtained is certainly uniquely
determined. Nevertheless, Proposition \ref{pro:cuspidal beta1} shows
that there is a canonical 1-1 correspondence (given simply by induction)
between on the one hand irreducible representations of $G(\psi_{\beta_{1}})$
which contain $\psi_{\beta_{1}}$ and which are non-trivial on $G(\psi_{\beta_{1}})\cap U_{2,2}$,
and on the other hand cuspidal representations of $G$ with $\beta_{1}$
in their respective orbits. We shall now extend this result to cuspidal
representations which have $\beta_{2}$ in their respective orbits,
and thus cover all cuspidal representations of $G$.

\subsection{The irregular cuspidal representations}

Assume now that $\pi$ is an irreducible representation of $G$ whose
orbit contains $\beta_{2}$. Although $\beta_{2}$ is not regular,
it is strongly semisimple in the sense of \cite{MR1311772},
Definition 3.1, and thus $\pi$ can be constructed explicitly in a
way similar to the regular case. More precisely, Proposition 3.3 in
\cite{MR1311772} implies that there exists an irreducible
representation $\tilde{\psi}_{\beta_{2}}$ of $G(\psi_{\beta_{2}})$,
such that $\tilde{\psi}_{\beta_{2}}|_{N_{1}}=\psi_{\beta_{2}}$, and
any extension of $\psi_{\beta_{2}}$ to $G(\psi_{\beta_{2}})$ is
of the form $\rho:=\theta\otimes\tilde{\psi}_{\beta_{2}}$, for some
irreducible representation $\theta$ pulled back from a representation
of $G(\psi_{\beta_{2}})/N_{1}$. Then \[
\pi=\Ind{G(\psi_{\beta_{2}})}{G}\rho\]
is an irreducible representation, any representation of $G$ with
$\beta_{2}$ in its orbit is of this form, and as in the regular case,
$\rho$ is uniquely determined by $\pi$. We then have a result completely
analogous to the previous proposition:

\begin{prop}
The representation $\pi$ is cuspidal if and only if $\rho$ does
not contain the trivial representation of $G(\psi_{\beta_{2}})\cap U_{2,2}$.
\end{prop}
\begin{proof}
The proof of Proposition \ref{pro:cuspidal beta1} with $\beta_{1}$
replaced by $\beta_{2}$, goes through up to the point where (under
the assumption that $\pi$ is not cuspidal) we get $\bar{x}p\in C_{G_{1^{4}}}(\left(\begin{smallmatrix}\eta & 0\\
0 & \eta\end{smallmatrix}\right))=G(\psi_{\beta_{2}})/N_{1}$. It then follows that $x\in G(\psi_{\beta_{2}})P_{2,2}$, and since
$U_{2,2}$ is normal in $P_{2,2}$, and $\langle\rho|_{G(\psi_{\beta_{2}})\cap\leftexp{x}{U_{2,2}}},\mathbf{1}\rangle$
only depends on the right coset of $x$ modulo $G(\psi_{\beta_{2}})$,
we get \[
0\neq\langle\rho|_{G(\psi_{\beta_{2}})\cap\leftexp{x}{U_{2,2}}},\mathbf{1}\rangle=\langle\rho|_{G(\psi_{\beta_{2}})\cap U_{2,2}},\mathbf{1}\rangle.\]

\end{proof}

\end{document}